\documentclass{article}
\usepackage{graphicx}
\usepackage{float}
\usepackage{mdframed}
\usepackage[most]{tcolorbox}
\usepackage{amsmath}
\usepackage{amsthm}
\usepackage{amsfonts}
\usepackage{amssymb}
\usepackage{bm}
\usepackage{mathrsfs}
\usepackage[utf8]{inputenc}
\usepackage{geometry}
\usepackage{orcidlink}
\usepackage[url=false]{biblatex}
\usepackage[font={small,it}]{subcaption}
\usepackage{hyperref}
\numberwithin{equation}{section}
\newtheorem{thrm}{Theorem}[section]
\newtheorem{lmm}[thrm]{Lemma}

\newtheorem{rmrk}[thrm]{Remark}

\newcommand{\bu}{\bm{u}}

\newcommand{\bk}{\bm{K}}

\newcommand{\bv}{\bm{v}}

\newcommand{\dotp}{\dot{p}}
\newcommand{\dotu}{\dot{\bm{u}}}

\newcommand{\telta}{\overline{\delta}}

\newcommand{\tilu}{\tilde{\bm{u}}_n^h}
\newcommand{\tilp}{\tilde{p}_n^h}

\newcommand{\enu}{e_n^{\bm{u}}}
\newcommand{\inu}{a^{\bm{u}}_n}
\newcommand{\inp}{a^p_n}
\newcommand{\enp}{e_n^p}

\newcommand{\frun}{\delta_t \bm{u}_n}
\newcommand{\frunh}{\delta_t \bm{u}_n^h}
\newcommand{\frpnh}{\delta_t p_n^h}

\newcommand{\franu}{\delta_t a_n^{\bm{u}}}
\newcommand{\frenu}{\delta_t e_n^{\bm{u}}}

\newtcolorbox{problem}[1][]{
    colback=white,
    colframe=black,
    boxrule=0.5pt,
    sharp corners,
    title=Problem (Q),
}

\newtcolorbox{discreteproblem}[1][]{
    colback=white,
    colframe=black,
    boxrule=0.5pt,
    sharp corners,
    title=Problem (Q-D),
}

\newtcolorbox{simpleproblem}[1][]{
    colback=white,
    colframe=black,
    boxrule=0.5pt,
    sharp corners,
    title=Problem (Q-D-S),
}

\newcommand{\change}[1]{#1}

\begin{document}

\title{Numerical analysis of the Biot equations coupled to frictional contact mechanics}
\author{Marius Nevland%
\thanks{Center for Modeling of Coupled Subsurface Dynamics, Department of Mathematics, University of Bergen, Allégaten 41, Bergen, 5020, Norway.}\
\thanks{Corresponding author: marius.nevland@uib.no}\ ,
Kundan Kumar\footnotemark[1]\ ,
Inga Berre\footnotemark[1]\ ,
Jakub Wiktor Both\footnotemark[1]\ ,
Eirik Keilegavlen\footnotemark[1]
}
\date{}

\setlength{\parindent}{0em}
\setlength{\parskip}{1em}
\maketitle
\begin{abstract}
We consider a mathematical model of a \change{poroviscoelastic} medium subject to frictional contact with a rigid obstacle, and study its numerical approximation. This model couples the Biot equations and contact conditions in the form of normal compliance and Coulomb friction. The resulting variational problem consists of a linear partial differential equation coupled to a nonlinear variational inequality. We propose and analyze a fully discrete numerical scheme for this problem, using conformal finite elements in space and the implicit Euler method in time. Existence and uniqueness of the discrete solution is established, and stability and a priori error estimates are derived. A numerical experiment is performed in which numerical error estimates are computed and compared to the theoretical results.
\end{abstract}
\section{Introduction}

In this paper, we consider a mathematical model of a linearly viscoelastic, porous body that is saturated with a fluid (also called a \change{\textit{poroviscoelastic}} body), and which is subject to contact conditions including friction. This can be viewed as an extension of classical models of linear elasticity subject to contact conditions, introducing a coupling of such models to fluid flow, resulting in a multiphysics problem.

Problems of linear elasticity subject to contact conditions lead to so-called variational inequalities \cite{duvant2012inequalities,kikuchi_oden}. The modeling of contact with friction typically leads to two types of constraints. The first type of constraint restricts the movement of the elastic body in the normal direction, and it is often formulated as a Karush-Kuhn-Tucker type inequality \cite{kikuchi_oden}, in which impenetrability between the body and obstacle is enforced (also called a Signorini condition). An alternative model is the normal compliance model, in which there is a constitutive relationship between the normal traction and the penetration \cite{martins_oden}. The second type of constraint regards the tangential motion of the body, modeling the tendency of the body to either stick or slip when there is contact, with the Coulomb friction law being a common model. 
%We shall restrict our attention to dynamic friction laws, where the slip is measured continuously using time derivatives, resulting in a time-dependent variational inequality. This is a natural choice when the contact conditions are coupled to a time-dependent flow problem.

The model of a linearly elastic body subject to the Signorini condition and the Coulomb friction law is notoriously difficult to analyze \cite{chouly2023finite}, with the central caveat being that this combination results in only a quasi-variational inequality \cite{quasi_variational}. However, if the Signorini condition is replaced by a normal compliance law, we obtain a classical variational inequality, which is more tractable for analysis. For this combination of the normal compliance and Coulomb friction laws, existence and uniqueness can be proven, given that a viscoelastic term is included in the stress tensor \cite{martins_oden}; this term is necessary to control the time derivative of displacement appearing in the Coulomb friction law. Numerical analysis of this model, including error estimates and convergence rates, have also successfully been conducted in works by Sofonea and Han \cite{hemivariational_ineq,han2002quasistatic}. Using the implicit Euler method, they obtained first order convergence in time, while the spatial convergence order is reduced by a square root. The latter is common in contact problems, due to the contact conditions reducing the regularity of the solution.

While there is a large body of literature on numerical analysis of an elastic body subject to contact conditions, including several books written on the subject \cite{chouly2023finite,han2002quasistatic,kikuchi_oden}, the case of a poroelastic body subject to contact conditions remains relatively unexplored. Coupled flow and elasticity in porous media (excluding contact mechanics) is frequently modeled by the Biot equations \cite{biot1941general}, which is a system of linear PDEs coupling linear elasticity to Darcy flow. Analysis of the Biot equations has been the subject of several works, including existence, uniqueness and regularity of the continuous problem \cite{showalter2000diffusion} and numerical analysis of finite element approximations, including the derivation of optimal order a priori error estimates \cite{murad1994stability,phillips_discrete,wheeler_l2}. 

There are a few works on numerical analysis of poroelasticity with contact conditions, in the context of mixed-dimensional models of fractured porous media, where the contact between the two sides of a fracture is modeled. In an article of Bonaldi et al. \cite{bonaldi_multiphase}, a two-phase, mixed-dimensional poromechanical model was considered, together with the Signorini condition and Coulomb friction law, and the stability and existence of a solution to their numerical scheme was proven. As previously mentioned, this combination of contact conditions causes significant difficulties, and as a result, the convergence to the continuous solution as the mesh size and time steps tend to zero was not proven. Another work of Bonaldi et al. \cite{bonaldi2024numerical} considered mixed-dimensional poromechanics with a Signorini condition, but without friction; this simplified contact model allowed for a proof of the convergence of the discrete solution to the continuous solution in the limit. Neither of the works included a priori error estimates between the continuous and discrete solutions. To the authors' knowledge, the only existing work that includes a priori error estimates is a recent paper by Bertrand and Banz \cite{banz2025contact}. In this paper, a numerical analysis of both a standard Galerkin method and an $hp$-FEM was performed on a time-independent variant of the Biot system subject to Signorini contact conditions, but without friction. In the case of the $hp$-FEM, they derive convergence rates that are optimal for the pressure, while they are reduced for the displacement. However, there does not appear to be any work on a priori error estimates for a space-time discretization of the time-dependent Biot equations subject to frictional contact constraints.

In this paper, we bridge this research gap, by presenting a numerical analysis on a fully discrete scheme for the time-dependent Biot equations in a \change{poroviscoelastic} medium, subject to contact constraints including normal compliance and Coulomb friction. There are two main reasons for our choice of this particular model. First, the well-posedness of a \change{poroviscoelastic} model subject to these contact constraints (generalized to the aforementioned mixed-dimensional setting of fractured porous media), has recently been proven by de Hoop and Kumar \cite{dehoop2023}, and the well-posedness of our model easily follows from their work. Second, as previously mentioned, the combination of viscoelasticity, normal compliance and Coulomb friction for the mechanical subproblem has been shown to be tractable for numerical analysis, including a priori error estimates.

Our numerical scheme uses conformal finite elements in space and the implicit Euler method in time. Existence, uniqueness and stability of the discrete solution is established, and a priori error estimates and corresponding convergence rates are proven. The existence of the discrete solution is proven by a novel method, in which we combine a fixed-point argument from the contact mechanics literature with the fixed-stress splitting scheme, the latter being a common way of solving the Biot equations by decoupling the flow and mechanics and iterating between the two problems \cite{kim2011stability,mikelic2013convergence}. Here, the fixed-stress splitting scheme is not used as a practical algorithm, as is typical, but rather as a theoretical tool to make the analysis feasible. The techniques used to derive the error estimates are a combination of those used by Girault et al. \cite{poroelastic_notmixed, Girault2019} and by Phillips and Wheeler \cite{phillips_discrete} for the Biot equations (the former two being in a mixed-dimensional setting) and the works of Sofonea and Han \cite{hemivariational_ineq,han2002quasistatic} on time-dependent variational inequalities. A central difficulty when considering poroelastic problems is the presence of a term accounting for the coupling of flow and mechanics. However, it will be shown that by combining the techniques of the aforementioned papers, this coupling term cancels in the stability and error analysis, and hence does not pose a problem.

The paper is structured as follows. We start by introducing the mathematical model for coupled flow and mechanics including frictional contact mechanics in Section \ref{sec:model}. In Section \ref{sec:variational_form}, we put the model in variational form, and assert the existence and uniqueness of a solution to this variational problem. In Section \ref{sec:discretization}, a fully discrete numerical scheme is proposed and analyzed. A numerical experiment is presented in Section \ref{sec:num_exp}, in which numerical convergence orders are obtained and compared to the theoretical results of Section \ref{sec:discretization}. Finally, concluding remarks are provided in Section \ref{sec:conclusion}.

\section{Model formulation} \label{sec:model}

Let $\Omega \subset \mathbb{R}^d$, $d=2$ or 3, be an open, bounded, connected domain with Lipschitz boundary $\Gamma=\partial \Omega$ and outwards unit normal $\bm{\nu}$, and denote its closure by $\overline{\Omega}$. This domain represents the interior of a viscoelastic, porous body that will potentially come into contact with a rigid obstacle. Within this domain, we model coupled flow and deformation by the quasi-static Biot equations for a viscoelastic, homogeneous, isotropic, porous solid saturated with a slightly compressible viscous fluid. The unknown variables to be solved for are the pressure $p$ and displacement $\bu$; time derivatives of these variables are denoted by a dot, $\dot{p}$ and $\dot{\bu}$.

Balance of momentum of the solid reads
\begin{equation} \label{eq:mom_bal}
    -\nabla \cdot \bm{\sigma}^{\text{por}}=\bm{f} \ \ \ \text{in} \ \ \ \Omega \times (0,T),
\end{equation}
where $\bm{\sigma}^{\text{por}}$ denotes the total \change{poroviscoelastic} stress tensor and $\bm{f}$ denotes body forces. The constitutive relation for $\bm{\sigma}^{\text{por}}$ is
\begin{equation} \label{eq:stress_total}
    \bm{\sigma}^{\text{por}}(\bm{u},p)=\bm{\sigma}(\bm{u})-\alpha p \bm{I},
\end{equation}
where $\bm{I}$ is the identity matrix, $\bm{u}$ is the displacement of the solid, $p$ is the fluid pressure, $\alpha>0$ is the dimensionless Biot coefficient, and $\bm{\sigma}$ is the viscoelastic stress tensor, following the Kelvin-Voigt model \cite{Sofonea_Matei_2012}:
\begin{equation} \label{stress_mechanics}
    \bm{\sigma}(\bm{u})=2G\bm{\varepsilon}(\bm{u})+\lambda(\nabla \cdot \bm{u})\bm{I}+\gamma\bm{\varepsilon}(\dotu)+\gamma(\nabla \cdot \dotu)\bm{I},
\end{equation}
where $\lambda$ and $G$ are the Lamé constants , $\gamma>0$ is a viscoelastic parameter, and $\bm{\varepsilon}(\bm{u})$ is the linearized strain tensor:
\begin{equation} \label{eq:strain_tensor}
    \bm{\varepsilon}(\bm{u})=\frac{1}{2}\left(\nabla\bm{u}+(\nabla\bm{u})^T
\right).
\end{equation}
The volumetric fluid flux follows Darcy's law (for simplicity, we neglect gravitational effects):
\begin{equation} \label{darcy}
    \bm{v}=-\frac{\bm{K}}{\mu}\nabla p,
\end{equation}
where $\bm{K}$ is the absolute permeability tensor, assumed to be symmetric, bounded, uniformly positive definite in space and constant in time, and $\mu>0$ is the constant fluid viscosity. The linearized fluid mass balance equation reads
\begin{equation} \label{eq:biot_mass_balance}
   c_0 \dot{p}+\alpha\nabla \cdot \dot{\bm{u}}+\nabla \cdot \bv=\psi \ \ \ \text{in} \ \ \ \Omega \times (0,T),
\end{equation}
where $c_0>0$ is the specific storage coefficient, and $\psi$ is a mass source or sink term.

The classical Biot problem consists of finding a pressure field $p$ and a displacement field $\bu$ that solve equations \eqref{eq:mom_bal} and \eqref{eq:biot_mass_balance}. We extend this problem by imposing frictional contact conditions between the porous body and the rigid obstacle. To this end, we introduce a subset $\Gamma_c \subset \Gamma$ of the boundary that potentially comes into contact with the rigid obstacle. For any vector-valued function $\bv$ defined on $\Gamma_c$, we define its normal and tangential components by
\begin{equation}
    v_{\nu}=\bv \cdot \bm{\nu}, \ \ \bv_{\tau}=\bv-v_{\nu}\bm{\nu},
\end{equation}
respectively. We denote by $\bm{t}$ the effective contact traction on $\Gamma_c$, defined as 
\begin{equation} \label{eq:contact_traction}
    \bm{t}=\bm{\sigma}^{\text{por}}\cdot \bm{\nu}.
\end{equation}
For the normal component of the contact traction, we impose a normal compliance law, similar to the one used by Martins and Oden \cite{martins_oden}. It introduces a constitutive relationship between the normal contact traction and the penetration, and we formulate it as follows:
\begin{equation} \label{eq:regularized_normal_compliance}
    t_{\nu} = -c_p({u}_{\nu}-g)_+ \ \ \text{on} \ \ \Gamma_c\times(0,T),
\end{equation}
where $c_p>0$ is a constant, which can be interpreted as the normal stiffness of the interface, $u_{\nu}$ is the normal component of the displacement, $g$ is the normal gap between the body and the obstacle in the undeformed configuration, and $(\cdot)_+$ denotes the positive cut of the function. 

The Coulomb friction law states that during contact, the body slips if the tangential contact traction reaches a friction bound, while no slipping occurs (the body sticks to the obstacle) if the bound is not reached. Moreover, this friction bound is proportional to the magnitude of the normal contact traction, with the proportionality constant being the friction coefficient, $\mu_{\text{fr}}$. Taking into account the normal compliance law \eqref{eq:regularized_normal_compliance} for the normal contact traction, the Coulomb friction law, imposed on $\Gamma_c \times (0,T)$, reads:
\begin{equation} \label{eq:friction_law}
\begin{aligned}
&\text{If} \ u_{\nu} \leq g \ \text{then} \ \ \lvert \bm{t}_{\tau}\rvert =0,\\
   & \text{If} \ u_{\nu} >g  \ \text{then} \ \begin{cases}
        \lvert \bm{t}_{\tau}\rvert \leq \mu_{\text{fr}} c_p(u_{\nu}-g),\\
        \lvert \bm{t}_{\tau}\rvert < \mu_{\text{fr}} c_p(u_{\nu}-g)  \implies \dot{\bu}_{\tau}=\bm{0}, \\
        \lvert \bm{t}_{\tau}\rvert = \mu_{\text{fr}} c_p(u_{\nu}-g) \implies \exists \ \zeta \geq 0 \ , \ \dot{\bu}_{\tau}=-\zeta \bm{t}_{\tau}.
    \end{cases}
\end{aligned}
\end{equation}
Here $\bm{t}_{\tau}$ and $\bu_{\tau}$ denote the tangential components of the contact traction and displacement, respectively.

Suitable boundary conditions are prescribed to close the system. To this end, we introduce two different decompositions of the boundary into disjoint subsets, $\Gamma=\Gamma_p\cup\Gamma_f$ and $\Gamma=\Gamma_d\cup\Gamma_n\cup\Gamma_c$, corresponding to the flow and mechanical boundary conditions, respectively. Both $\Gamma_d$ and $\Gamma_p$ are assumed to have positive measure. For simplicity, we set the boundary conditions to be homogeneous, with the exception of the mechanical Neumann boundary condition:
\begin{equation}
\begin{aligned}
    &p=0 \ \ \ \text{on} \ \ \Gamma_p\times (0,T), \ \  \bk\nabla p\cdot\bm{\nu}=0 \ \ \text{on} \ \ \Gamma_f\times (0,T),\\
    &\bu=\bm{0} \ \ \  \text{on} \ \ \Gamma_d\times (0,T), \ \ \bm{\sigma}^{\text{por}}\cdot\bm{\nu}=\bm{b} \ \ \text{on} \ \ \Gamma_n\times(0,T),
\end{aligned}
\end{equation}
Finally, initial values of the displacement and pressure are prescribed, i.e. $\bu(0)=\bu_0$, $p(0)=p_0$. In practice, one often ensures compatibility between these values, for instance by solving \eqref{eq:mom_bal} with prescribed $p=p_0$.

\section{Variational formulation} \label{sec:variational_form}

In this section, the model of Section \ref{sec:model} is put in variational form, and existence and uniqueness of solutions to the variational form is asserted. We start by reviewing the functional-analytic notation to be used in this paper.

\subsection{Notation and functional setting} \label{sec:notation}
Denote the space of second order symmetric tensors on $\mathbb{R}^d$ by $\mathbb{S}^d$. The usual $L^2$ inner product will be used for functions with values in both $\mathbb{R}$, $\mathbb{R}^d$ and $\mathbb{S}^d$. In all cases, we use parentheses $(\cdot,\cdot)$ to denote this inner product over $\Omega$. In cases where the integration domain is not $\Omega$, we specify the domain by a subscript; for instance, $(\cdot,\cdot)_{\Gamma_n}$ denotes the $L^2$ inner product over $\Gamma_n$. \change{If the output space is $\mathbb{R}^d$ or $\mathbb{S}^d$, we denote the corresponding $L^2$-space by a boldface, $\bm{L}^2(\Omega)$.} However, for simplicity, we drop the boldface notation when denoting the $L^2$-norm, and use the same notation $\lVert \cdot \lVert_{L^2(\Omega)}$ for functions with values in both $\mathbb{R}$, $\mathbb{R}^d$ and $\mathbb{S}^d$.

% We use standard notation for the canonical inner products on $\mathbb{R}^d$ and $\mathbb{S}^d$, namely,
% \begin{equation} \label{eq:inner_products}
% \begin{aligned}
%     &\bm{u} \cdot \bm{v}=\sum_i u_i v_i \ \ \forall \bm{u}, \bm{v} \in \mathbb{R}^d, \\
%     &\bm{\sigma} : \bm{\tau} = \sum_i \sum_j \sigma_{ij}\tau_{ij} \ \ \forall \bm{\sigma}, \bm{\tau}\in \mathbb{S}^d.
% \end{aligned}
% \end{equation}
% The $L^2$ inner product will be used for functions defined on $\Omega$ with values in both $\mathbb{R}$, $\mathbb{R}^d$ and $\mathbb{S}^d$. We use parentheses $(\cdot,\cdot)$ to denote the inner product in all cases, with the following meanings:
% \begin{equation} \label{eq:l2_inner_products}
% \begin{aligned}
%     &(p,q)=\int_{\Omega}pq \ dx \ \ \forall p,q \in L^2(\Omega;\mathbb{R}), \\
%     &(\bm{u},\bm{v})=\int_{\Omega}\bm{u} \cdot \bm{v} \ dx \ \ \forall \bm{u}, \bm{v} \in L^2(\Omega;\mathbb{R}^d), \\
%     &(\bm{\sigma},\bm{\tau})=\int_{\Omega} \bm{\sigma} : \bm{\tau} \ dx \ \ \forall \bm{\sigma},\bm{\tau} \in L^2(\Omega;\mathbb{S}^d).
% \end{aligned}
% \end{equation}
We let $H^s(\Omega)$, $1\leq s<\infty$ denote the usual Sobolev spaces, consisting of functions whose weak derivatives up to order $s$ lie in $L^2(\Omega)$. We denote by $\lVert \cdot \rVert_{H^s(\Omega)}$ the norm induced from the usual inner product on these spaces, and we denote by $\lvert \cdot \rvert_{H^s(\Omega)}$ the usual seminorm. \change{Again we use boldface, $\bm{H}^s(\Omega)$, to denote the $H^s$-space of functions with values in either $\mathbb{R}^d$ or $\mathbb{S}^d$, but use the same notation $\lVert \cdot \rVert_{H^{s}(\Omega)}$ and $\lvert \cdot \rvert_{H^s(\Omega)}$ for the $H^s$-norms for functions with values in both $\mathbb{R}, \mathbb{R}^d$ and $\mathbb{S}^d$.}

Let $\gamma: H^1(\Omega) \to L^2(\Gamma)$ denote the trace operator, which satisfies the trace inequality:
\begin{equation} \label{eq:trace_inequality}
    \forall v \in H^1(\Omega), \ \ \ \lVert \gamma v \rVert_{L^2(\Gamma)} \leq C_{\tau}\lVert v \rVert_{H^1(\Omega)},
\end{equation}
for some $C_{\tau}>0$ depending only on $\Omega$ and $\Gamma$. For simplicity, we shall denote the trace $\gamma v$ of a function $v \in H^1(\Omega)$ simply by $v$. Next, we define the following closed subspace of $H^1(\Omega)$:
\begin{equation}
    Q=\{ q\in H^1(\Omega) \ ; \ q=0 \ \mathrm{on} \ \Gamma_p\}.
\end{equation}
which is the space for the pressure. Similarly, the space for the displacement is the following closed subspace of $\bm{H}^1(\Omega)$:
\begin{equation} \label{eq:func_space_displacement}
    \bm{V}=\{\bm{v} \in \bm{H}^1(\Omega) \ ; \ \bv=\bm{0}\ \mathrm{on} \ \Gamma_d \},
\end{equation}
We shall also need the dual of this space, which we denote by $\bm{V}^*$. Duality pairings are denoted by chevrons, $\langle \cdot, \cdot \rangle$. Since $\Gamma_p$ and $\Gamma_d$ have positive measures, we may use the Poincaré's inequality on the spaces $Q$ and $\bm{V}$, and Korn's first inequality on the space $\bm{V}$. Poincaré's inequality on $Q$ reads: There exists a constant $C_{\Gamma}$ depending only on $\Omega$ and $\Gamma$ such that
\begin{equation} \label{eq:poincare_inequality}
    \forall v \in Q, \ \ \lVert v \rVert_{L^2(\Omega)} \leq C_{\Gamma}\lvert v \rvert_{H^1(\Omega)},
\end{equation}
with an equivalent inequality also holding for $\bv \in \bm{V}$. Korn's first inequality reads: There exists a constant $C_{\kappa}$ depending only on $\Omega$ and $\Gamma$ such that
\begin{equation} \label{eq:korn_inequality}
    \forall \bm{v} \in \bm{V}, \ \ \lvert \bm{v} \rvert_{H^1(\Omega)} \leq C_{\kappa}\lVert \bm{\varepsilon}(\bm{v})\rVert_{L^2(\Omega)}.
\end{equation}
\change{Time-dependent $L^p$-spaces (or Bochner spaces), for functions defined on a time interval $(a,b)$ with values in a space $X$, are defined as follows:
\begin{equation} \label{eq:bochner_space}
    L^p(a,b;X)=\left\{ f \ \text{measurable in} \ (a,b) \ ; \ \int_a^b\lVert f(t) \rVert_X^p dt < \infty\right\},
\end{equation}
equipped with the norm
\begin{equation} \label{eq:norm_bochner}
    \lVert f \rVert_{L^p(a,b;X)}=\left(\int_a^b \lVert f(t) \rVert_X^p dt\right)^{1/p},
\end{equation}
with the usual modification if $p=\infty$}. We shall also need the spaces $H^1(a,b;X)$ and $H^2(a,b;X)$, defined as follows:
\begin{equation}
    H^1(a,b;X)= \{ f\in L^2(a,b;X) \ ; \ \dot{f} \in L^2(a,b;X)\}.
\end{equation}
\begin{equation}
    H^2(a,b;X)= \{ f\in H^1(a,b;X) \ ; \ \ddot{f} \in L^2(a,b;X)\}.
\end{equation}
equipped with the norms
\begin{equation} \label{eq:norm_h1_bochner}
    \lVert f \rVert_{H^1(a,b;X)}=\left(\int_a^b \lVert f(t) \rVert_X^2 + \lVert \dot{f}(t) \rVert^2_X dt\right)^{1/2},
\end{equation}
\begin{equation} \label{eq:norm_h2_bochner}
    \lVert f \rVert_{H^2(a,b;X)}=\left(\int_a^b \lVert f(t) \rVert_X^2 + \lVert \dot{f}(t) \rVert^2_X + \lVert \ddot{f}(t) \rVert^2_X dt\right)^{1/2}.
\end{equation}
\change{Finally, dual spaces are denoted by a prime, e.g., $Q^{\prime}$ is the dual space of $Q$, and angular brackets $\langle \cdot , \cdot \rangle$ are used to denote duality pairings between spaces defined over $\Omega$.}

\subsection{Variational formulation} \label{sec:variational_form_actually}

We begin by defining some operators to be used in the variational formulation. The usual elastic and viscoelastic bilinear forms $A[\cdot,\cdot]: \bm{V} \times \bm{V} \to \mathbb{R}$ and $B[\cdot,\cdot]: \bm{V} \times \bm{V} \to \mathbb{R}$ are defined as follows
\begin{equation} \label{eq:bilinear_form_1}
    A[\bm{u},\bm{v}]=2G(\bm{\varepsilon}(\bm{u}),\bm{\varepsilon}(\bm{v}))+\lambda(\nabla \cdot \bm{u},\nabla \cdot \bm{v}),
\end{equation}
\begin{equation} \label{eq:bilinear_form_2}
    B[\bm{u},\bm{v}]=\gamma(\bm{\varepsilon}(\bm{u}),\bm{\varepsilon}(\bm{v}))+\gamma(\nabla \cdot \bm{u},\nabla \cdot \bm{v}).
\end{equation}
It is well known that these bilinear forms are continuous and coercive with respect to the $H^1$-norm \cite{ciarlet}:
\change{\begin{equation} \label{eq:continuity_bilinear}
\begin{aligned}
&\exists \  m_A>0, m_B>0 \ \text{such that} \ \forall \ \bm{u},\bm{v} \in \bm{V}, \\
    &\lvert A[\bm{u},\bm{v}] \rvert \leq m_A\lVert \bm{u} \rVert_{H^1(\Omega)}\lVert \bm{v} \rVert_{H^1(\Omega)}, \ \lvert B[\bm{u},\bm{v}]\rvert \leq m_B\lVert \bm{u} \rVert_{H^1(\Omega)}\lVert \bm{v} \rVert_{H^1(\Omega)},
\end{aligned}
\end{equation}
\begin{equation} \label{eq:v-ellipticity}
    \exists \  \alpha_A>0,\alpha_B>0 \ \text{such that} \ \forall \ \bm{v}\in \bm{V}, \   A[\bm{v},\bm{v}]\geq\alpha_A\lVert \bm{v}\rVert^2_{H^1(\Omega)}, \ B[\bm{v},\bm{v}]\geq \alpha_B\lVert\bm{v}\rVert^2_{H^1(\Omega)}. 
\end{equation}}
We also define two nonlinear functionals, $J_{\nu}[\cdot,\cdot]: \bm{V} \times \bm{V} \to \mathbb{R}$ and $J_{\tau}[\cdot,\cdot]: \bm{V} \times \bm{V} \to \mathbb{R}$, representing the contributions from the normal compliance law and the Coulomb friction law, respectively:
\begin{equation} \label{eq:j_norm_functional}
    J_{\nu}[\bm{u},\bm{v}]=\int_{\Gamma_c}c_p(u_{\nu}-g)_+v_{\nu} ds,
\end{equation}
\begin{equation} \label{eq:j_tang_functional}
    J_{\tau}[\bm{u},\bm{v}]=\int_{\Gamma_c}\mu_{\text{fr}} c_p(u_{\nu}-g)_+\lvert\bm{v}_{\tau}\rvert ds,
\end{equation}
and we collect these into a single nonlinear functional $J[\cdot,\cdot]: \bm{V} \times \bm{V} \to \mathbb{R}$:
\begin{equation} \label{eq:j_functional}
    J[\bm{u},\bm{v}]=J_{\nu}[\bu,\bv]+J_{\tau}[\bu,\bv].
\end{equation}
Finally, given sufficient regularity of $\bm{f}$ and $\bm{b}$, we define $\bm{f}^*\in \bm{V}^*$ as follows:
\begin{equation} \label{eq:source_mechanics_dual}
    \langle \bm{f}^*,\bv\rangle = (\bm{f},\bv)+(\bm{b},\bv)_{\Gamma_n} \ \  , \ \  \bv \in \bm{V}.
\end{equation}
The conformal variational problem, henceforth referred to as problem (Q), then reads as follows: 
\begin{problem}
Given $\bm{f} \in L^2(0,T;\bm{L}^2(\Omega))$, $\bm{b} \in L^2(0,T;\bm{L}^2(\Gamma_n))$, $\psi \in L^2(0,T;L^2(\Omega))$ and $g \in L^2(\Gamma_c)$, find \change{$\bm{u} \in H^1(0,T;\bm{V})$} and \change{$p \in L^2(0,T;Q)\cap H^1(0,T;Q^{\prime})$} such that, for a.e. $t \in (0,T)$:
\begin{equation} \label{eq:weak_mechanics}
\begin{aligned}
    &A[\bm{u},\bm{v}-\dotu]+B[\dotu,\bm{v}-\dotu]-\alpha( p,\nabla \cdot (\bm{v}-\dotu))+ J[\bm{u},\bm{v}]-J[\bm{u},\dotu] \geq \langle\bm{f}^*,\bm{v}-\dotu \rangle \ \ \forall\bm{v} \in \bm{V},
\end{aligned}
\end{equation}
\begin{equation} \label{eq:weak_biot}
    \change{\langle c_0\dotp,q\rangle+(\alpha\nabla \cdot \dotu,q)}+\frac{1}{\mu}(\bk\nabla p, \nabla q)=(\psi,q) \ \ \forall q \in Q,
\end{equation}
subject to the initial conditions:
\begin{align}
    \bu\lvert_{t=0} &= \bu_0, \label{eq:init_cond1}\\
    (c_0p+\alpha\nabla \cdot \bu)\lvert_{t=0} &= c_0p_0+\alpha \nabla \cdot \bu_0. \label{eq:init_cond2}
\end{align}
\end{problem}
\change{Note that the modeling of a viscoelastic solid allows the displacement to be taken in a stronger temporal space,
compared to a purely elastic solid.}

The derivation of \eqref{eq:weak_biot} follows the usual strategy of assuming a sufficiently regular solution to equation \eqref{eq:biot_mass_balance} in strong form, multiplying by test functions and using Green's formula. The derivation of the variational inequality \eqref{eq:weak_mechanics} from equation \eqref{eq:mom_bal} and the contact conditions \eqref{eq:regularized_normal_compliance}, \eqref{eq:friction_law} is also classical; see e.g. chapter 10 of \cite{kikuchi_oden}.

\subsection{Existence and uniqueness of solutions} \label{sec:existence_sol}

Regarding the existence and uniqueness of a solution $(\bu,p)$ to problem (Q), belonging to the specified function spaces, we first note that existence and uniqueness have been proven for mathematical models of similar structure in the papers of Martins and Oden \cite{martins_oden} and de Hoop and Kumar \cite{dehoop2023}. In \cite{martins_oden}, the existence and uniqueness of solutions to a viscoelastic contact problem, with normal compliance and Coulomb friction, is proven. This proof is expanded in \cite{dehoop2023} to show the existence and uniqueness of solutions to the dynamic Biot equations in a fractured porous medium, including flow in fractures and fracture contact mechanics. The existence and uniqueness proof for problem (Q) will be nearly identical in structure to the proof of \cite{dehoop2023}. Therefore, rather than explicitly writing the proof, we instead point out the differences between the mathematical models considered in their work and in ours, and the corresponding modifications of their proof to accommodate our case. The techniques of the existence proof of \cite{dehoop2023} are fairly standard: The friction term is regularized, Galerkin approximations are constructed and passed to the limit, and compactness and monotonicity arguments are used to establish convergence of the nonlinear contact terms.

There are three differences between the model used in \cite{dehoop2023} and our model. First, they do not consider inhomogeneous Neumann conditions in their analysis; however, this change has no effect on the analysis. Secondly, they are considering the dynamic instead of the quasi-static Biot problem, meaning that an acceleration term of the form $\ddot{\bu}$ is included in their momentum balance equation. The inclusion of this term allows the velocity $\dotu$ to be taken in the slightly stronger space $L^{\infty}(0,T;\bm{L}^2(\Omega))\cap L^2(0,T;\bm{V})$, compared to our case of just $\dotu \in L^2(0,T;\bm{V})$. However, the space $L^2(0,T;\bm{V})$ is in fact sufficient to establish the aforementioned compactness estimates that are needed to prove convergence of the nonlinear contact terms (see page 423 of \cite{martins_oden} for the exact arguments), which is the only part of the existence proof that requires careful consideration.

The third difference is that they are modeling a fractured medium, which involves modeling the frictional contact mechanics between the two sides of a fracture embedded in the \change{poroviscoelastic} medium, rather than our case of contact between a \change{poroviscoelastic} medium and a rigid obstacle. Additionally, they are considering Darcy flow in the fracture as well as in the surrounding porous medium, resulting in an additional mass balance equation in the fracture, as well as additional coupling terms. However, if the fracture pressure is assumed to be constant, then by symmetry, the fractured medium model of \cite{dehoop2023} can be reduced to a model of contact between a \change{poroviscoelastic} medium and a rigid obstacle. Hence, the proof of \cite{dehoop2023} is straightforward to adapt to our case. We conclude with the following Lemma:
\begin{lmm}
There exists a unique solution $(\bu,p)$ to problem (Q), belonging to the specified function spaces.
\end{lmm}
 
% The final difference between the model in \cite{dehoop2023} and our model is that they are considering the dynamic Biot problem, meaning that an acceleration term of the form $\ddot{\bu}$ was included in the equations. The inclusion of this term allows the velocity $\dotu$ to be taken in the slightly stronger space $L^{\infty}(0,T;L^2(\Omega))\cap L^2(0,T;\bm{V})$, compared to our case of just $\dotu \in L^2(0,T;\bm{V})$. However, the space $L^2(0,T;\bm{V})$ is all that is needed to establish the convergence of the nonlinear contact terms, which is the part of the proof that requires careful consideration.

\section{Discretization} \label{sec:discretization}

In this section, we will present a fully discrete scheme for numerically approximating solutions to the variational formulation \eqref{eq:weak_mechanics}, \eqref{eq:weak_biot}, using conforming finite elements in space and the implicit Euler method in time. We choose $N$ equidistant time points in the interval $[0,T]$. The time step size is accordingly $\Delta t=\frac{T}{N}$ and the discrete time points are $t_n=n\Delta t, \ 0 \leq n \leq N$. For any time-dependent function $v(\cdot,t)$, assumed to be continuous in time, we write $v_n:=v(\cdot,t_n)$. The backward finite difference approximation of the time derivative is denoted as follows:
\begin{equation} \label{eq:time_difference}
    \delta_t v_n :=\frac{v_n-v_{n-1}}{\Delta t}.
\end{equation}
Let $\mathcal{T}^h$ denote a regular family of conforming triangulations of $\overline{\Omega}$, indexed by the maximal diameter $h$ of the elements of $\mathcal{T}^h$. Let $k_1\geq1$ and $k_2\geq1$ denote two polynomial orders. We consider standard Lagrangian finite elements, where the space $\bm{V}$ of displacements and the space $Q$ of pressures are approximated by $\bm{V}^h$ and $Q^h$, respectively, defined as follows:
\begin{align} \label{eq:v_h}
    \bm{V}^h&=\left\{ \bv^h \in \bm{V} : \forall \ T\in\mathcal{T}^h, \ \bv^h\rvert_{T}\in \mathbb{P}^d_{k_1}(T) \right\},\\
    Q^h &= \left\{ q^h\in Q : \forall \ T\in\mathcal{T}^h, \ q^h\rvert_{T}\in \mathbb{P}_{k_2}(T) \right\},
\end{align}
where $\mathbb{P}_p(T)$ denotes the space of polynomials of order $p$ on $T$ ($\mathbb{P}_p^d(T)$ denotes the Cartesian product of $d$ such spaces). We also introduce the standard Lagrange interpolation operators $I_{\bm{V}^h}:\bm{V}\to \bm{V}^h$ and $I_{Q^h}: Q\to Q^h$, and use the notation $\bu^I := I_{\bm{V}^h}(\bu)$, $p^I:=I_{Q^h}(p)$. These satisfy the usual interpolation error estimates \cite{ciarlet}:
\begin{equation} \label{eq:interpolation_estimate}
    \lVert\bm{u}-\bu^I \rVert_{H^1(\Omega)}\leq C_1h^{r}\lvert\bu \rvert_{H^{r+1}(\Omega)}, \ \lVert p-p^I \rVert_{H^1(\Omega)} \leq C_2h^{s}\lvert p \rvert_{H^{s+1}(\Omega)},
\end{equation}
with $1 \leq r \leq k_1$ and $1 \leq s \leq k_2$, given that $\bu \in \bm{H}^{r+1}(\Omega)$ and $p \in H^{s+1}(\Omega)$. 

%Moreover, given sufficient regularity of the trace of the displacement, we have the following useful interpolation estimate on the contact boundary, again for $1 \leq s_1 \leq k_1$:
% \begin{equation} \label{eq:interp_error_boundary}
%     \lVert \bu-\bu^I\rVert_{L^2(\Gamma_c)} \leq C h^{s_1+1}\lvert \bu \rvert_{H^{s_1+1}(\Gamma_c)}, 
% \end{equation}
% given that the trace $\bu \in H^{s_1+1}(\Gamma_c)$.

\subsection{Fully discrete approximation} \label{sec:fully_discrete_form}
The following fully discrete variational problem is proposed, called problem (Q-D).

\begin{discreteproblem}
Assume $\bm{f}\in H^1(0,T;\bm{L}^2(\Omega)), \bm{b} \in H^1(0,T;\bm{L}^2(\Gamma_n))$, $\psi\in H^1(0,T;L^2(\Omega))$ and $g \in L^2(\Gamma_c)$.\\

At time $t=0$, given initial data $p_0\in Q$ and $\bu_0 \in \bm{V}$, set $p^h_0=p_0^I$ and $\bu_0^h=\bu_0^I$. For any $n \in \{1,2,..,N\}$, given $\bm{u}_{n-1}^h, p_{n-1}^h$, find $\bm{u}_n^h \in \bm{V}^h$ and $p_n^h \in Q^h$ such that:
\begin{equation} \label{eq:discrete_mechanics}
\begin{aligned}
    A[\bm{u}_n^h,\bm{v}^h-\frunh]+B[\frunh,\bm{v}^h&-\frunh]-\alpha\left(p_n^h,\nabla \cdot (\bm{v}^h-\frunh)\right) \\
    &+J[\bm{u}_n^h,\bm{v}^h]-J[\bm{u}_n^h,\frunh] \geq \left\langle \bm{f}^*_n,\bm{v}^h-\frunh\right\rangle \ \ \forall \bm{v}^h \in \bm{V}^h,
\end{aligned}
\end{equation}
\begin{equation} \label{eq:discrete_biot}
    \left( c_0\frpnh+\alpha\nabla \cdot \frunh, q^h\right)+\frac{1}{\mu}\left(\bk\nabla p_n^h,\nabla q^h\right)=\left(\psi_n,q^h\right) \ \ \forall q^h \in Q^h,
\end{equation}
\end{discreteproblem}
The assumed regularity in time of the data $\bm{f}, \bm{b}$ and $\psi$ ensures that $\bm{f}_n^*$ and $\psi_n$ are well-defined functions in $\bm{V}^*$ and $L^2(\Omega)$, respectively.

\subsection{Existence and uniqueness of the discrete solution} \label{eq:existence_discrete}

We begin this section by proving an estimate on the $J$-functionals, which is needed in the upcoming proof of existence and uniqueness of the discrete solution. This estimate will also become useful in Section \ref{sec:error_estimates} when deriving error estimates between the discrete and continuous solutions. Similar types of estimates can be found in the book of Sofonea and Han \cite{han2002quasistatic}, where they are used in the numerical analysis of several types of frictional contact problems.
\begin{lmm} \label{lemma:friction_estimate}
    The $J$-functional \eqref{eq:j_functional} satisfies the following estimate:
    \begin{equation}
        J[\bu_1,\bv_2]-J[\bu_1,\bv_1]+J[\bu_2,\bv_1]-J[\bu_2,\bv_2] \leq C_J \lVert \bu_1-\bu_2 \rVert_{H^1(\Omega)}\lVert \bv_1-\bv_2 \rVert_{H^1(\Omega)},
    \end{equation}
    for all $\bu_1,\bu_2,\bv_1,\bv_2 \in \bm{V}$, with $C_J=C_{\tau}^2(1+\mu_{\textnormal{fr}})c_p$.
\end{lmm}
\begin{proof}
We estimate:
\begin{equation*} \label{eq:estimate_friction}
\begin{aligned}
    &\lvert J[\bm{u}_1,\bv_2]-J[\bm{u}_1,\bv_1]+J[\bm{u}_2,\bv_1]-J[\bm{u}_2,\bv_2]\rvert\\
    &=\left\lvert \int_{\Gamma_c} c_p ((\bm{u}_{1\nu}-g)_+-(\bm{u}_{2\nu}-g)_+)(\bv_{2\nu}-\bv_{1\nu}) ds+\int_{\Gamma_c}\mu_{\text{fr}} c_p ((\bm{u}_{1\nu}-g)_+-(\bm{u}_{2\nu}-g)_+)(\lvert \bm{v}_{2\tau} \rvert-\lvert\bm{v}_{1\tau}\rvert)ds \right \rvert \\
    &\leq c_p\lVert (\bm{u}_{1\nu}-g)_+-(\bm{u}_{2\nu}-g)_+ \rVert_{L^2(\Gamma_c)} \lVert\bv_{2\nu}-\bv_{1\nu}\rVert_{L^2(\Gamma_c)}+\mu_{\text{fr}} c_p \lVert (\bm{u}_{1\nu}-g)_+-(\bm{u}_{2\nu}-g)_+ \rVert_{L^2(\Gamma_c)}\lVert\bv_{2\tau}-\bv_{1\tau}\rVert_{L^2(\Gamma_c)}\\
    &\leq (1+\mu_{\text{fr}})c_p\lVert \bu_1-\bu_2\lVert_{L^2(\Gamma_c)}\lVert \bv_1-\bv_2\lVert_{L^2(\Gamma_c)}\\
    &\leq C_{\tau}^2(1+\mu_{\text{fr}})c_p\lVert \bu_1-\bu_2\lVert_{H^1(\Omega)}\lVert \bv_1-\bv_2\lVert_{H^1(\Omega)}.
\end{aligned}
\end{equation*}
The first inequality is due to the Cauchy-Schwarz and triangle inequalities. For the second inequality, we use that the map $\bu_n \mapsto (\bu_n-g)_+$ is Lipschitz continuous (with Lipschitz constant equal to one), in addition to the facts that $\lVert \bu_n \rVert \leq \lVert \bu \rVert$ and $\lVert \bu_{\tau} \rVert \leq \lVert \bu \rVert$. The trace inequality \eqref{eq:trace_inequality} then gives the third and final inequality.
\end{proof}

Our strategy for proving the existence of a solution to problem (Q-D) is as follows:

1. We first consider a simpler problem, where the contact mechanics terms are lagged. Existence of solutions to this simper problem is proven by constructing a fixed-stress splitting scheme, and showing the contraction of the scheme.

2. We then substitute back the full contact terms, and apply a fixed-point argument from the contact mechanics literature.

The simplified problem we will consider, which shall be called problem (Q-D-S), is the following: 
\begin{simpleproblem}
Assume $\bm{f}\in H^1(0,T;\bm{L}^2(\Omega)), \bm{b} \in H^1(0,T;\bm{L}^2(\Gamma_n))$, $\psi\in H^1(0,T;L^2(\Omega))$ and $g\in L^2(\Gamma_c)$.\\

At time $t=0$, given initial data $p_0\in Q$ and $\bu_0 \in \bm{V}$, set $p^h_0=p_0^I$ and $\bu_0^h=\bu_0^I$. For any $n \in \{1,2,..,N\}$, given $\bm{u}_{n-1}^h, p_{n-1}^h$, and given $\bm{\eta}\in \bm{V}^h$, find $\bm{u}_n^h \in \bm{V}^h$ and $p_n^h \in Q^h$ such that:
\begin{equation} \label{eq:discrete_mechanics_simple}
\begin{aligned}
    \left(2G\bm{\varepsilon}(\bm{u}_n^h)+\gamma\bm{\varepsilon}(\frunh),\bm{\varepsilon}(\bm{v}^h-
    \frunh)\right)&+\left(\lambda\nabla \cdot \bm{u}_n^h+\gamma\nabla \cdot \frunh-\alpha p_n^h,\nabla \cdot (\bm{v}^h-\frunh)\right) \\
    &+J[\bm{\eta},\bm{v}^h]-J[\bm{\eta},\frunh] \geq \left\langle\bm{f}_n^*,\bm{v}^h-\frunh\right\rangle \ \ \forall \bm{v}^h \in \bm{V}^h.
\end{aligned}
\end{equation}
\begin{equation} \label{eq:discrete_biot_again}
    \left( c_0\frpnh+\alpha\nabla \cdot \frunh, q^h\right)+\frac{1}{\mu}\left(\bk\nabla p_n^h,\nabla q^h\right)=\left(\psi_n,q^h\right) \ \ \forall q^h \in Q^h,
\end{equation}
\end{simpleproblem}
In other words, we keep the discrete mass balance equation \eqref{eq:discrete_biot}, but replace the discrete variational inequality \eqref{eq:discrete_mechanics} by a simpler version in which the first argument of the $J$-functional is a known, given function $\bm{\eta} \in \bm{V}^h$. This corresponds to the normal contact traction being prescribed on $\Gamma_c$, leading to the so-called Tresca friction model, which is much easier to analyze \cite{chouly2023finite,kikuchi_oden}. The function $\bm{\eta}$ may be chosen differently for different time steps, but for simplicity, we omit the subscript $n$. Note that we have substituted the definition of the bilinear forms \eqref{eq:bilinear_form_1}, \eqref{eq:bilinear_form_2} here, which is convenient when proving the contraction of the upcoming fixed-stress splitting scheme.

% Keep the discrete mass balance equation \eqref{eq:discrete_biot}, but replace the discrete variational inequality \eqref{eq:discrete_mechanics} by
% \begin{equation} \label{eq:discrete_mechanics_simple}
% \begin{aligned}
%     (2G\bm{\varepsilon}(\bm{u}_n^h)+\gamma\bm{\varepsilon}(\frac{\delta\bm{u}_n^h}{\Delta t}),\bm{\varepsilon}(\bm{v}^h-\frac{\delta \bm{u}^h_n}{\Delta t}))&+(\lambda(\nabla \cdot \bm{u}_n^h)+\gamma(\nabla \cdot (\frac{\delta \bm{u}_n^h}{\Delta t}))-\alpha p_n^h,\nabla \cdot (\bm{v}^h-\frac{\delta \bm{u}_n^h}{\Delta t})) \\
%     &+J[\bm{\eta},\bm{v}^h]-J[\bm{\eta},\frac{\delta\bm{u}_n^h}{\Delta t}] \geq \langle\bm{f}_n^*,\bm{v}^h-\frac{\delta\bm{u}_n^h}{\Delta t}\rangle \ \ \forall \bm{v}^h \in \bm{V}^h,
% \end{aligned}
% \end{equation}
% for some given $\bm{\eta} \in \bm{V}^h$. 

The fixed-stress splitting scheme is a well-known algorithm for solving problems of coupled flow and mechanics, by decoupling the flow and mechanics parts and iterating between the two problems \cite{kim2011stability,mikelic2013convergence}. The convergence of the scheme for the fully discrete Biot equations has been shown by Almani et al. \cite{almani2017convergence,almani2016multi}, and a recent paper by Almani and Kumar \cite{fixed_stress_contact} extended the convergence analysis to also include frictionless Signorini contact conditions imposed on parts of the boundary. The fixed-stress splitting scheme developed herein, and its subsequent analysis, closely follows \cite{fixed_stress_contact}.

We use the subscript $k$ to denote the iteration of the splitting scheme. The fixed-stress splitting scheme for problem (Q-D-S) reads:

\textbf{Step 1: Flow}. Given $p_{n,k}^h\in Q^h$ and $\bu_{n,k}^h \in \bm{V}^h$, find $p^h_{n,k+1} \in Q^h$ such that
\begin{equation} \label{eq:fixed_stress_biot}
\begin{aligned}
    \forall q^h \in Q^h \ , \ \left( (c_0+\frac{\alpha^2}{\lambda})\delta_tp_{n,k+1}^h+\alpha\nabla \cdot \delta_t\bu_{n,k}^h, q^h\right)&+\frac{1}{\mu}\left(\bk\nabla p_{n,k+1}^h,\nabla q^h\right)\\
    &=\left(\frac{\alpha^2}{\lambda}\delta_tp_{n,k}^h,q^h\right)+\left(\psi_n,q^h\right),
\end{aligned}
\end{equation}

\textbf{Step 2: Mechanics}. Given $\bm{\eta} \in \bm{V}^h$ and given $p_{n,k+1}^h$ from step 1, find $\bm{u}^h_{n,k+1} \in \bm{V}^h$ such that
\begin{equation} \label{eq:fixed_stress_mechanics}
\begin{aligned}
    \forall \bm{v}^h \in \bm{V}^h \ , \ &\left(2G\bm{\varepsilon}(\bm{u}^h_{n,k+1})+\gamma\bm{\varepsilon}(\delta_t\bu_{n,k+1}^h),\bm{\varepsilon}(\bm{v}^h-\delta_t\bu_{n,k+1}^h)\right)\\
    &+\left(\lambda\nabla \cdot \bm{u}^h_{n,k+1}+\gamma\nabla \cdot \delta_t\bu_{n,k+1}^h-\alpha p^h_{n,k+1},\nabla \cdot (\bm{v}^h-\delta_t\bu_{n,k+1}^h)\right)\\
    &\ \ \ \ \ \ \ \ \ \ \ \ \ \ \ \ \ \ \ \ +J[\bm{\eta},\bm{v}^h]-J[\bm{\eta},\delta_t\bu_{n,k+1}^h] \geq \left\langle\bm{f}_n^*,\bm{v}^h-\delta_t\bu_{n,k+1}^h\right\rangle,
\end{aligned}
\end{equation}
and the scheme is initialized for time step $n\geq1$ by setting $p_{n,0}^h=p_{n-1}^h$ and $\bm{u}_{n,0}^h=\bm{u}_{n-1}^h$. As is standard for the fixed-stress splitting scheme, the stabilization term $\alpha^2\lambda^{-1}\delta_tp_{n,k+1}^h$ has been added to the left-hand side of equation \eqref{eq:fixed_stress_biot}, with a similar term on the right-hand side for consistency.

\begin{lmm} \label{lemma:fixed_stress_existence}
    There exists a unique solution to each step of the splitting scheme.
\end{lmm}

\begin{proof}
The first step, \eqref{eq:fixed_stress_biot}, corresponds to a square system of linear equations in finite dimension. Hence, the existence and uniqueness of the solution to \eqref{eq:fixed_stress_biot} follows by showing that if all data are zero (including data from the previous time step and splitting scheme iteration), then the only solution is the zero solution. This can be shown by testing \eqref{eq:fixed_stress_biot} with $q^h=p^h_{n,k+1}$, resulting in, when all data are zero:
\begin{equation}
    (c_0+\frac{\alpha^2}{\lambda}\frac{1}{\Delta t})\lVert p_{n,k+1}^h \rVert_{L^2(\Omega)}^2+\frac{1}{\mu} \lVert \bk^{1/2}\nabla p_{n,k+1}^h\rVert_{L^2(\Omega)}^2=0,
\end{equation}
from which existence and uniqueness follow. For the second step, we note that, since the pressure is fixed during this step, \eqref{eq:fixed_stress_mechanics} is a standard variational inequality of the second kind, as defined in chapter 4 of \cite{han2002quasistatic}. Hence, the existence of a unique solution to \eqref{eq:fixed_stress_mechanics} is inferred from a result in \cite{han2002quasistatic}, which states that a variational inequality of the second kind has a unique solution if the bilinear forms are continuous and coercive, and if the $J$-functional is proper, convex and lower semi-continuous. The assumptions on the bilinear forms hold (recall \eqref{eq:continuity_bilinear}, \eqref{eq:v-ellipticity}), and the assumptions on the $J$-functional are also well-
known to hold given our definition \eqref{eq:j_functional} of this functional \cite{han2002quasistatic}. The existence and uniqueness of the solution follows.
\end{proof}

We proceed to derive an estimate on the differences in iterations of the splitting scheme, which will lead to a contraction result. We shall denote differences in iterations of the splitting scheme by $\telta$, to distinguish it from the notation $\delta_t$ used to denote differences between time steps. Hence, for some variable $v^h_{n,k}$, depending on the time step $t_n$, the mesh $\mathcal{T}^h$ and the splitting scheme iteration $k$, we define:
\begin{equation} \label{eq:difference_iterates_splitting}
    \telta v^{h,k}_{n}=v^{h,k}_{n}-v^{h,k-1}_{n}.
\end{equation}
Following standard nomenclature, we define the volumetric mean stress
\begin{equation} \label{eq:mean_stress}
    \sigma^{v}_{n,k}=\lambda\nabla \cdot \bu_{n,k}^h-\alpha p_{n,k}^h,
\end{equation}
and its difference between fixed-stress iterations is accordingly
\begin{equation} \label{eq:mean_stress_diff}
\telta\sigma^{v}_{n,k}=\lambda\nabla \cdot \telta\bm{u}^h_{n,k}-\alpha\telta p^h_{n,k}.
\end{equation}
For notational convenience, the following constant is defined:
\begin{equation} \label{eq:beta_constant}
    \beta=\frac{c_0}{\alpha^2}+\frac{1}{\lambda}.
\end{equation}
\begin{lmm} \label{lemma:fixed_stress_contraction}
    The iterations of the fixed-stress splitting scheme \eqref{eq:fixed_stress_biot}, \eqref{eq:fixed_stress_mechanics} satisfy the following estimate:
    \begin{equation} \label{eq:contraction_estimate}
\begin{aligned}
    &\left(4G\lambda+\frac{2\gamma\lambda}{\Delta t}\right)\lVert \bm{\varepsilon}(\telta \bu^h_{n,k+1})\rVert_{L^2(\Omega)}^2+\left(\lambda^2+\frac{2\gamma\lambda}{\Delta t}\right)\lVert \nabla \cdot \telta \bm{u}^h_{n,k+1}\rVert_{L^2(\Omega)}^2\\
    &+\frac{2\Delta t}{\beta\mu}\lVert \bm{K}^{1/2}\nabla\telta p_{n,k+1}^h\rVert_{L^2(\Omega)}^2+\lVert \telta \sigma_{n,k+1}^v\rVert_{L^2(\Omega)}^2\leq \frac{1}{(\beta\lambda)^2}\lVert \telta \sigma^{v}_{n,k}\rVert_{L^2(\Omega)}^2.
\end{aligned}
\end{equation}
\end{lmm}

\begin{proof}
We start by computing the difference between the flow equations \eqref{eq:fixed_stress_biot} at iterations $k$ and $k+1$, testing with $q^h=\telta p_{n,k+1}^h$ in both cases. The result is, after multiplying both sides by $\Delta t$:
\begin{equation}
\begin{aligned}
    \beta\lVert \alpha\telta p^h_{n,k+1}\rVert^2_{L^2(\Omega)}+\frac{\Delta t}{\mu}\lVert \bk^{1/2}\nabla\telta p_{n,k+1}^h\rVert_{L^2(\Omega)}^2&=-\left(\frac{\alpha}{\lambda}\telta \sigma^{v}_{n,k},\telta p^h_{n,k+1}\right)\\
    &\leq\frac{\varepsilon}{2}\lVert \alpha \telta p^h_{n,k+1}\rVert_{L^2(\Omega)}^2+\frac{1}{2\varepsilon \lambda^2}\rVert \telta \sigma^{v}_{n,k}\rVert_{L^2(\Omega)}^2,
\end{aligned}
\end{equation}
where the Cauchy-Schwarz and Young's inequalities have been used to obtain the second line. We choose $\varepsilon=\beta$ to get
\begin{equation} \label{eq:fixed_stress_estimate_biot}
    \lVert \alpha\telta p^h_{n,k+1}\rVert^2_{L^2(\Omega)}+\frac{2\Delta t}{\beta\mu}\lVert \bk^{1/2}\nabla\telta p_{n,k+1}^h\rVert_{L^2(\Omega)}^2\leq \frac{1}{(\beta \lambda)^2}\lVert \telta \sigma^{v}_{n,k}\rVert_{L^2(\Omega)}^2.
\end{equation}
Next, we consider the variational inequality \eqref{eq:fixed_stress_mechanics}. In this case, we test \eqref{eq:fixed_stress_mechanics} for iteration $k$ with $\bm{v}^h=\delta_t\bu_{n,k+1}^h$, and for iteration $k+1$ with $\bm{v}^h=\delta_t\bu_{n,k}^h$, and add the two inequalities. Since the first arguments of the $J$-functionals are fixed, this choice of test functions causes the $J$-functionals to cancel when the inequalities are added. The resulting estimate becomes:
\begin{equation} \label{eq:fixed_stress_estimate_mechanics_almost}
\begin{aligned}
    &\left(2G+\frac{\gamma}{\Delta t}\right)\lVert \bm{\varepsilon}(\telta \bm{u}^h_{n,k+1})\rVert_{L^2(\Omega)}^2+\left(\lambda+\frac{\gamma}{\Delta t}\right)\lVert \nabla \cdot \telta \bm{u}^h_{n,k+1}\rVert_{L^2(\Omega)}^2-\alpha(\telta p^h_{n,k+1},\nabla \cdot \telta \bm{u}^h_{n,k+1})\leq 0.
\end{aligned}
\end{equation}
Note that the inequality sign has been switched. We multiply both sides of \eqref{eq:fixed_stress_estimate_mechanics_almost} by $2\lambda$, which will enable us to complete the square on the coupling term, and add it to \eqref{eq:fixed_stress_estimate_biot} to obtain:
% \begin{equation} \label{eq:fixed_stress_estimate_mechanics}
% \begin{aligned}
%     &(4G\lambda+\frac{2\gamma\lambda}{\Delta t})\lVert \bm{\varepsilon}(\telta \bm{u}^h_{n,k+1})\rVert_{L^2(\Omega)}^2+(2\lambda^2+\frac{2\gamma\lambda}{\Delta t})\lVert \nabla \cdot \telta \bm{u}^h_{n,k+1}\rVert_{L^2(\Omega)}^2-2\alpha\lambda(\telta p^h_{n,k+1},\nabla \cdot \telta \bm{u}^h_{n,k+1})\leq 0.
% \end{aligned}
% \end{equation}
% Adding \eqref{eq:fixed_stress_estimate_biot} and \eqref{eq:fixed_stress_estimate_mechanics}, we obtain:
\begin{equation} \label{eq:fixed_stress_full_estimate1}
\begin{aligned}
    &\left(4G\lambda+\frac{2\gamma\lambda}{\Delta t}\right)\lVert \bm{\varepsilon}(\telta \bu^h_{n,k+1})\rVert_{L^2(\Omega)}^2+\left(\lambda^2+\frac{2\gamma\lambda}{\Delta t}\right)\lVert \nabla \cdot \telta \bm{u}^h_{n,k+1}\rVert_{L^2(\Omega)}^2+\frac{2\Delta t}{\beta\mu}\lVert \bm{K}^{1/2}\nabla\telta p_{n,k+1}^h\rVert_{L^2(\Omega)}^2\\
    &+\left\{\lVert \alpha \telta p^h_{n,k+1} \rVert_{L^2(\Omega)}^2-2\alpha\lambda(\telta p^h_{n,k+1},\nabla \cdot \telta \bm{u}^h_{n,k+1})+\lambda^2\lVert \nabla \cdot \telta \bm{u}^h_{n,k+1}\rVert_{L^2(\Omega)}^2 \right\}\leq \frac{1}{(\beta \lambda)^2}\lVert \telta \sigma^{v}_{n,k}\rVert_{L^2(\Omega)}^2.
\end{aligned}
\end{equation}
We recognize the term inside the curly brackets as $\lVert \telta \sigma_{n,k+1}^v\rVert_{L^2(\Omega)}^2$, and the estimate \eqref{eq:contraction_estimate} follows.
\end{proof}
Next, we establish convergence of the sequences generated by the fixed-stress splitting scheme, and prove that the limit functions solve the simplified problem (Q-D-S).

\begin{lmm} \label{lemma:sol_simple}
There exist limit functions $p_n^h\in Q^h$ and $\bu_n^h \in \bm{V}^h$ \change{such that the iterations $p_{n,k}^h$ and $\bu_{n,k}^h$ of the fixed-stress splitting scheme satisfy, as $k \to \infty$:
\begin{equation*}
    p_{n,k}^h \to p_n^h \ \text{in} \ H^1(\Omega), \ \bu_{n,k}^h \to \bu_n^h \ \text{in} \ \bm{H}^1(\Omega),
\end{equation*}
with strong convergence in the above norms.} Moreover, the pair $(\bu_n^h, p_n^h)$ is a solution to problem (Q-D-S) at time step $n$.
\end{lmm}
\begin{proof}
The preceding estimate, Lemma \ref{lemma:fixed_stress_contraction}, can be used to establish convergence of the fixed-stress iterations $p_{n,k}^h, \bu_{n,k}^h$ as $k\to\infty$. We start by proving the convergence of the pressure. Since it follows from the definition of $\beta$ that $\beta\lambda>1$, the estimate \eqref{eq:contraction_estimate} defines a contraction.
From this it follows that, in particular, $\nabla \telta p_{n,k}^h$ converges geometrically to zero, implying that $\nabla p_{n,k}^h$ is a Cauchy sequence in $\change{\bm{L}^2(\Omega)}$, and by the Poincaré inequality \eqref{eq:poincare_inequality}, we have that $p_{n,k}^h$ is also a Cauchy sequence in $L^2(\Omega)$. By the completeness of this space, $p_{n,k}^h\to p_n^h \in L^2(\Omega)$, from which it follows that $\nabla p_{n,k}^h \to \nabla p_n^h\in \change{\bm{L}^2(\Omega)}$. Hence, $p_n^h \in H^1(\Omega)$, and since $p_{n,k}^h\in Q^h$ for all $k$, and $Q^h$ is a closed subspace of $H^1(\Omega)$, we conclude that $p_n^h \in Q^h$. Similarly for the displacement, we have from \eqref{eq:contraction_estimate} that $\bm{\varepsilon}(\telta \bu_{n,k}^h)$ converges geometrically to zero, implying that $\bm{\varepsilon}(\bu_{n,k}^h)$ is a Cauchy sequence in $\change{\bm{L}^2(\Omega)}$. By combining the Poincaré inequality \eqref{eq:poincare_inequality} and the Korn inequality \eqref{eq:korn_inequality}, this implies that $\bu_{n,k}^h$ is a Cauchy sequence in $\bm{H}^1(\Omega)$, and by completeness, $\bu_{n,k}^h\to \bu_n^h \in \bm{H}^1(\Omega)$. Since $\bu_{n,k}^h \in \bm{V}^h$ for all $k$, and $\bm{V}^h$ is a closed subspace of $\bm{H}^1(\Omega)$, we conclude that $\bu_n^h \in \bm{V}^h$.
%from which it follows that $\sigma^v_{n,k}$ and $\nabla \cdot \bu_{n,k}^h$ are Cauchy sequences in $L^2(\Omega)$. By completeness, both sequences converge to elements of this space, and it then follows from the definition \eqref{eq:mean_stress} of $\sigma_{n,k}^v$ that $p_{n,k}^h \to p_n^h\in L^2(\Omega)$. Moreover, the estimate \eqref{eq:contraction_estimate} implies that $\nabla p_{n,k}^h$ is a Cauchy sequence in $L^2(\Omega)$, hence the sequence converges to an element of this space. Together with the fact that $p_{n,k}^h \to p_n^h$, this implies that $\nabla p_{n,k}^h \to \nabla p_n^h \in L^2(\Omega)$.

Turning to problem (Q-D-S), it is straightforward to argue that $p_n^h$ and $\bu_n^h$ solve the discrete mass balance equation \eqref{eq:discrete_biot_again}, as this equation only involves linear, continuous operators. Hence, by letting $k \to \infty$ in step 1 of the splitting scheme \eqref{eq:fixed_stress_biot}, we immediately recover \eqref{eq:discrete_biot_again}. For the variational inequality \eqref{eq:fixed_stress_mechanics} of step 2 of the splitting scheme, the nonlinear contact term that depends on the solution (i.e. the second $J$-functional on the left-hand side) must be considered more carefully. By performing a similar estimate as done in Lemma \ref{lemma:friction_estimate}, one can show that the contact terms of the variational inequality \eqref{eq:discrete_mechanics_simple} of problem (Q-D-S) is obtained in the limit. For instance, we have for the friction term $J_{\tau}$:
\begin{equation}
\begin{aligned}
&\left \lvert J_{\tau}[\bm{\eta},\delta_t\bu_{n,k}^h]-J_{\tau}[\bm{\eta},\delta_t\bu_n^h] \right \rvert =\\
    &\left \lvert \int_{\Gamma_c}\mu_{\text{fr}} c_p(\bm{\eta}_{\nu}-g)_+\left\lvert (\delta_t\bu_{n,k}^h)_{\tau}\right\rvert ds-\int_{\Gamma_c}\mu_{\text{fr}} c_p(\bm{\eta}_{\nu}-g)_+\left\lvert (\delta_t \bu_n^h)_{\tau}\right\rvert ds\right \rvert\\
    &=\left \lvert \int_{\Gamma_c}\frac{\mu_{\text{fr}} c_p}{\Delta t}(\bm{\eta}_{\nu}-g)_+\left(\vphantom{\rule{0pt}{2.8ex}}\left\lvert (\bm{u}^h_{n,k}-\bm{u}^h_{n-1})_{\tau}\right\rvert-\left\lvert (\bm{u}^h_n-\bm{u}^h_{n-1})_{\tau}\right\rvert\right)\right \rvert ds\\
    &\leq \frac{\mu_{\text{fr}} c_p}{\Delta t}\lVert (\bm{\eta}_{\nu}-g)_+ \rVert_{L^2(\Gamma_c)}\lVert (\bm{u}^h_{n,k}-\bm{u}^h_n)_{\tau}\rVert_{L^2(\Gamma_c)},
\end{aligned}
\end{equation}
which goes to zero as $k \to \infty$, since $\bm{u}^h_{n,k}\to\bm{u}^h_n$. The same type of estimate can be done on the normal compliance term $\left \lvert J_{\nu}[\bm{\eta},\delta_t\bu_{n,k}^h]-J_{\nu}[\bm{\eta},\delta_t\bu_n^h] \right \rvert$. For the other terms of \eqref{eq:fixed_stress_mechanics}, convergence to the respective terms of \eqref{eq:discrete_mechanics_simple} follows from linearity and continuity. Hence, we conclude that $p_{n}^h,\bu_{n}^h$ solves problem (Q-D-S).
\end{proof}

Finally, we are in a position to prove the existence and uniqueness of a solution to the original problem (Q-D).
\begin{thrm} \label{thm:existence_discrete}
    % If one of the following conditions are fulfilled:

    % 1. $\alpha_A\geq C_{\tau}^2(1+\mu_{\text{fr}})c_p$,

    % 2. $\alpha_A<C_{\tau}^2(1+\mu_{\text{fr}})c_p$ and $\Delta t < \frac{\alpha_B}{C_{\tau}^2(1+\mu_{\text{fr}})c_p-\alpha_A}$,
    
    If $\alpha_A+\frac{\alpha_B}{\Delta t}>C_J$ (which can always be fulfilled by choosing a sufficiently small time step), then (Q-D) is guaranteed to have a unique solution. In other words, there exists for any $n\in \{1,2,...,N\}$ a unique pair $(\bu_n^h,p_n^h)\in \bm{V}^h\times Q^h$ that solves \eqref{eq:discrete_mechanics} and \eqref{eq:discrete_biot}.
\end{thrm}

\begin{proof}
The existence of a solution is proven by combining the solution to problem (Q-D-S) obtained in Lemma \ref{lemma:sol_simple} with a fixed-point argument from the contact mechanics literature; see for instance \cite{Han1999}. We introduce the fixed point operator $\Lambda:\bm{V}^h \to \bm{V}^h$, defined by
\begin{equation} \label{eq:fixed_point_operator}
    \Lambda(\bm{\eta})=\bm{u}^{\bm{\eta}}_n,
\end{equation}
where $\bm{u}^{\bm{\eta}}_n$ denotes the displacement solution to problem (Q-D-S) for a given function $\bm{\eta}$ (i.e. the limit of the sequence $\bu_{n,k}^h$ of the fixed-stress splitting scheme \eqref{eq:fixed_stress_biot}, \eqref{eq:fixed_stress_mechanics} as $k\to\infty$, as obtained in Lemma \ref{lemma:sol_simple}). We similarly denote the corresponding pressure solution by $p_n^{\bm{\eta}}$. For ease of notation, we drop the $h$ superscript in this proof. We will show that this operator defines a contraction. 

Consider two solutions $(\bm{u}^{\bm{\eta_1}}_n, p^{\bm{\eta_1}}_n)$ and $(\bm{u}^{\bm{\eta_2}}_n, p^{\bm{\eta_2}}_n)$ to problem (Q-D-S), corresponding to different functions $\bm{\eta}_1$ and $\bm{\eta}_2$, respectively. We use a tilde to denote the differences between these solutions; $\tilde{p}_n=p_n^{\bm{\eta}_1}-p_n^{\bm{\eta}_2}$, $\tilde{\bm{u}}_n=\bm{u}^{\bm{\eta}_1}_n-\bm{u}^{\bm{\eta}_2}_n$ and $\tilde{\bm{\eta}}=\bm{\eta}_1-\bm{\eta}_2$. The simplified variational inequality \eqref{eq:discrete_mechanics_simple} corresponding to the $\bm{u}^{\bm{\eta}_1}_n$-solution is tested with $\bm{v}^h=\delta_t\bm{u}^{\bm{\eta}_2}_n$, and the inequality corresponding to the $\bm{u}^{\bm{\eta}_2}_n$-solution is tested with $\bm{v}^h= \delta_t\bm{u}^{\bm{\eta}_1}_n$, and the resulting equations are added. The discrete mass balance equation \eqref{eq:discrete_biot_again} is tested with $q^h=\tilde{p}_n$ and the equations for each solution are subtracted. The resulting estimate becomes (we now revert back to the notation of the bilinear forms \eqref{eq:bilinear_form_1} and \eqref{eq:bilinear_form_2}):
\begin{equation} \label{eq:fpi_arg_1}
\begin{aligned}
    &A[\tilde{\bm{u}}_n, \tilde{\bm{u}}_n]+\frac{1}{\Delta t}B[\tilde{\bm{u}}_n, \tilde{\bm{u}}_n]+c_0\lVert \tilde{p}_n \rVert_{L^2(\Omega)}^2+\frac{\Delta t}{\mu}\lVert \bm{K}^{1/2}\nabla\tilde{p}_n\rVert_{L^2(\Omega)}^2\\
    &\leq J[\bm{\eta}_1,\bm{u}^{\bm{\eta}_2}_n-\bm{u}_{n-1}]
    -J[\bm{\eta}_1,\bm{u}^{\bm{\eta}_1}_n-\bm{u}_{n-1}]
    +J[\bm{\eta}_2,\bm{u}^{\bm{\eta}_1}_n-\bm{u}_{n-1}]
    -J[\bm{\eta}_2,\bm{u}^{\bm{\eta}_2}_n-\bm{u}_{n-1}]\\
    &\leq C_J\lVert \tilde{\bm{\eta}}\rVert_{H^1(\Omega)}\lVert \tilde{\bm{u}}_n\rVert_{H^1(\Omega)},
\end{aligned}
\end{equation}
where Lemma \ref{lemma:friction_estimate} has been used for the final inequality. Together with the coercivity conditions \eqref{eq:v-ellipticity}, this implies the estimate
\begin{equation} \label{eq:fpi_arg_3}
    (\alpha_A+\frac{\alpha_B}{\Delta t})\lVert \tilde{\bm{u}}_n\rVert_{H^1(\Omega)}^2 \leq C_J\lVert \tilde{\bm{\eta}}\rVert_{H^1(\Omega)}\lVert \tilde{\bm{u}}_n\rVert_{H^1(\Omega)},
\end{equation}
or equivalently,
\begin{equation} \label{eq:contact_contraction_estimate}
    \lVert \tilde{\bm{u}}_n\rVert_{H^1(\Omega)} \leq \frac{C_J}{ \alpha_A+\frac{\alpha_B}{\Delta t}}\lVert \tilde{\bm{\eta}} \rVert_{H^1(\Omega)},
\end{equation}
Hence, the operator $\Lambda$ \eqref{eq:fixed_point_operator} is a contraction, and consequently has a unique fixed point, if
\begin{equation} \label{eq:time_step_constraint_1}
   \alpha_A+\frac{\alpha_B}{\Delta t}>C_J. 
\end{equation}
Denote this fixed point by $\bm{\eta}^*$, and let $(\bm{u}^*_n, p^*_n)$ denote the corresponding solution to problem (Q-D-S) with $\bm{\eta}^*$ as the given function. Then $(\bm{u}^*_n, p^*_n)$ is also a solution to problem (Q-D). Indeed, the only difference between problem (Q-D) and problem (Q-D-S) is the variational inequality, but due to the fixed-point property $\bm{\eta}^*=\bm{u}^*_n$, the solution to the simplified variational inequality \eqref{eq:discrete_mechanics_simple} coincides with that of the full variational inequality \eqref{eq:discrete_mechanics}. We thus conclude that there exists a solution to problem (Q-D), given that the time step size fulfills the condition \eqref{eq:time_step_constraint_1}.

Finally, we prove the uniqueness of the solution to problem (Q-D). This proof is straightforward at this point, as it follows the exact same steps as the preceding proof of the contraction of the operator $\Lambda$. Assuming there are two solutions, $(\bm{u}_{n1}^h,p_{n1}^h)$ and $(\bm{u}_{n2}^h,p_{n2}^h)$, to the discrete equations \eqref{eq:discrete_mechanics}, \eqref{eq:discrete_biot}, denoting the differences between the solutions by tildes, and performing the same computations as before, \eqref{eq:fpi_arg_1}-\eqref{eq:fpi_arg_3}, we obtain the estimate
\begin{equation} \label{eq:uniqueness_estimate}
    \left(\alpha_A+\frac{\alpha_B}{\Delta t}\right)\lVert \tilu\rVert_{H^1(\Omega)}^2+c_0\lVert \tilp \rVert_{L^2(\Omega)}^2+\frac{\Delta t}{\mu}\lVert \bk^{1/2}\nabla \tilp \rVert_{L^2(\Omega)}^2 \leq C_J\lVert \tilu \rVert_{H^1(\Omega)}^2,
\end{equation}
from which the uniqueness follows, given the fulfillment of the same condition as before, \eqref{eq:time_step_constraint_1}. Theorem \ref{thm:existence_discrete} follows.
\end{proof}

\subsection{Stability estimate} \label{sec:stability_estimates}

In this section, we derive a stability estimate for the solution $(\bu_n^h,p_n^h)$ to the discrete equations \eqref{eq:discrete_mechanics}, \eqref{eq:discrete_biot}.
\begin{thrm}
If the time step size satisfies $\Delta t<2\alpha_A\alpha_B$, then the solution $(\bu_n^h,p_n^h)$ to problem (Q-D) satisfies the following stability bound:
\begin{equation} \label{eq:stability_estimate}
    \begin{aligned}
        &\lVert \bu_N^h \rVert_{H^1(\Omega)}^2+\lVert p_N^h \rVert_{L^2(\Omega)}^2+ \sum_{n=1}^N \Delta t \lVert \frunh \rVert_{H^1(\Omega)}^2+\sum_{n=1}^N \Delta t \lVert \bk^{1/2}\nabla p_n^h \rVert_{L^2(\Omega)}^2\leq C,
    \end{aligned}
\end{equation}
where $C$ depends on the final time $T$, but is independent of $n, h$ and $\Delta t$.
\end{thrm}
\begin{proof}
We shall need the following discrete version of the chain rule, for some $v_n$ defined at the discrete time points $t_n$:
\begin{equation} \label{eq:discrete_chain_rule}
    (\delta_t v_n,v_n)=\frac{1}{2\Delta t}\left( \lVert v_n \rVert_{L^2(\Omega)}^2-\lVert v_{n-1}\rVert_{L^2(\Omega)}^2+\lVert v_n-v_{n-1}\rVert_{L^2(\Omega)}^2\right).
\end{equation}
The flow equation \eqref{eq:discrete_biot} is tested with $q^h=p_n^h$. Using the chain rule above, the obtain the following stability equality:
\begin{equation} \label{eq:stability_flow_1}
\begin{aligned}
    \frac{c_0}{2\Delta t}\left(\lVert p_n^h \rVert_{L^2(\Omega)}^2-\lVert p_{n-1}^h \rVert_{L^2(\Omega)}^2+\lVert  p_n^h-p_{n-1}^h \rVert_{L^2(\Omega)}^2\right)&+\alpha\left(p_n^h,\nabla \cdot \frunh\right)\\
    &+\frac{1}{\mu}\lVert \bk^{1/2}\nabla p_n^h \rVert_{L^2(\Omega)}^2=\left(\psi_n,p_n^h\right).
\end{aligned}
\end{equation}
Next, we test the variational inequality \eqref{eq:discrete_mechanics} with $\bv^h=\bm0$, and use the discrete chain rule \eqref{eq:discrete_chain_rule} on the bilinear form $A$, resulting in
\begin{equation} \label{eq:stability_mechanics_1}
    \begin{aligned}
        \frac{1}{2\Delta t}\left(A[\bu_n^h,\bu_n^h]\right.-A[\bu_{n-1}^h,\bu_{n-1}^h]&+A[\left. \bu_n^h-\bu_{n-1}^h, \bu_n^h-\bu_{n-1}^h]\right)+B[\frunh,\frunh]\\
        &\leq \alpha\left(p_n^h,\nabla \cdot \frunh\right)-J[\bu_n^h,\frunh]+\left\langle \bm{f}_n^*,\frunh\right\rangle.
    \end{aligned}
\end{equation}
Equations \eqref{eq:stability_flow_1} and \eqref{eq:stability_mechanics_1} are then added, and we get, noting that the coupling term cancels:
\begin{equation} \label{eq:stability_combined_1}
    \begin{aligned}
        &\frac{1}{2\Delta t}\left(A [\bu_n^h,\bu_n^h]-A[\bu_{n-1}^h,\bu_{n-1}^h]+A[ \bu_n^h-\bu_{n-1}^h, \bu_n^h-\bu_{n-1}^h]\right)+B[\frunh,\frunh]\\
        &+\frac{c_0}{2\Delta t}\left(\lVert p_n^h \rVert_{L^2(\Omega)}^2-\lVert p_{n-1}^h \rVert_{L^2(\Omega)}^2+\lVert p_n^h-p_{n-1}^h
        \rVert_{L^2(\Omega)}^2\right)\\
        &+\frac{1}{\mu}\lVert \bk^{1/2}\nabla p_n^h \rVert_{L^2(\Omega)}^2\leq \left(\psi_n,p_n^h\right)+\left\langle\bm{f}_n^*,\frunh\right\rangle-J[\bu_n^h,\frunh].
    \end{aligned}
\end{equation}
Next, we sum each side of \eqref{eq:stability_combined_1} from $n=1$ to $N$, multiply by $\Delta t$, and use the continuity \eqref{eq:continuity_bilinear} and coercivity \eqref{eq:v-ellipticity} of the bilinear forms to get
\begin{equation} \label{eq:stability_almost_final}
    \begin{aligned}
        &\frac{\alpha_A}{2}\left(\lVert\bu_N^h\rVert_{H^1(\Omega)}^2+\sum_{n=1}^N\lVert \bu_n^h-\bu_{n-1}^h\rVert_{H^1(\Omega)}^2\right)+\frac{c_0}{2}\left( \lVert p_N^h\rVert_{L^2(\Omega)}^2+\sum_{n=1}^N\lVert p_n^h-p_{n-1}^h \rVert_{L^2(\Omega)}^2\right)\\
        &+\alpha_B \sum_{n=1}^N \Delta t\lVert\frunh\rVert_{H^1(\Omega)}^2+\frac{1}{\mu}\sum_{n=1}^N\Delta t \lVert \bk^{1/2}\nabla p_n^h \rVert_{L^2(\Omega)}^2\\
        &\leq \frac{m_A}{2}\lVert\bu_0^h\rVert_{H^1(\Omega)}^2+\frac{c_0}{2}\lVert p_0^h \rVert_{L^2(\Omega)}^2 + \Delta t \sum_{n=1}^N\left( \left(\psi_n,p_n^h\right)+\left\langle \bm{f}_n^*,\frunh\right\rangle-J[\bu_n^h,\frunh]\right).
    \end{aligned}
\end{equation}
The first two terms inside the sum on the right-hand side are bounded by straightforward usage of the Cauchy-Schwarz and Young inequalities (for the second term, we additionally use the trace inequality on the resulting boundary integral $\lVert \delta_t \bu_n^h\rVert_{L^2(\Gamma_n)}$):
\begin{equation} \label{eq:stability_source_bounds_1}
    \left\lvert \left(\psi_n,p_n^h\right)\right\rvert \leq \varepsilon_1 \lVert p_n^h \rVert_{L^2(\Omega)}^2+C_1\rVert \psi_n \rVert_{L^2(\Omega)}^2,
\end{equation}
\begin{equation}
\label{eq:stability_source_bounds_2}
    \left\lvert \left\langle\bm{f}_n^*,\frunh\right\rangle\right\rvert \leq (\varepsilon_2+\varepsilon_3) \lVert \frunh \rVert_{H^1(\Omega)}^2 + C_2 \lVert \bm{f}_n \rVert_{L^2(\Omega)}^2+C_3\lVert \bm{b}_n \rVert_{L^2(\Gamma_n)}^2
\end{equation}
Here, $\varepsilon_i$ denote the constants of Young's inequality that are chosen freely, while the other constants are denoted by $C_i=C_i(\varepsilon_i)$.
For the nonlinear contact term on the right-hand side of \eqref{eq:stability_almost_final}, we first obtain from the Cauchy-Schwarz and triangle inequalities:
\begin{equation} \label{eq:stability_contact_1}
    \left\lvert J[\bu_n^h,\frunh]\right\rvert \leq c_p\lVert (\bu_{n\nu}^h-g)_+\rVert_{L^2(\Gamma_c)}\lVert (\frunh)_{\nu}\rVert_{L^2(\Gamma_c)}+\mu_{\text{fr}}c_p\lVert (\bu_{n\nu}^h-g)_+\rVert_{L^2(\Gamma_c)}\lVert (\frunh)_{\tau}\rVert_{L^2(\Gamma_c)}.
\end{equation}
Moreover, we have as an immediate consequence of the positive cut function $(\cdot)_+$ that 
\begin{equation}
    \lVert (\bu_{n\nu}^h-g)_+\rVert_{L^2(\Gamma_c)} \leq \lVert \bu_{n\nu}^h-g\rVert_{L^2(\Gamma_c)} \leq \lVert \bu_{n\nu}^h\rVert_{L^2(\Gamma_c)}+\lVert g \rVert_{L^2(\Gamma_c)}.
\end{equation}
Using this, in addition to the inequalities $\lVert \bv_{\change{\nu}} \rVert \leq \lVert \bv \rVert$ and $\lVert \bv_{\tau}\rVert \leq \lVert \bv \rVert$ (for any vector-valued function $\bv$ defined on $\Gamma_c$), we get from \eqref{eq:stability_contact_1} the bound:
\begin{equation} \label{eq:stability_contact_bound}
\begin{aligned}
   \left\lvert J[\bu_n^h,\frunh]\right\rvert &\leq C\left(\lVert \bu_n^h \rVert_{L^2(\Gamma_c)}+\lVert g \rVert_{L^2(\Gamma_c)}\right)\lVert \frunh \rVert_{L^2(\Gamma_c)} \\
   &\leq (\varepsilon_4+\varepsilon_5)\lVert \frunh \rVert_{H^1(\Omega)}^2+C_4\lVert \bu_n^h \rVert_{H^1(\Omega)}^2+C_5 \lVert g \rVert_{L^2(\Gamma_c)}^2,
\end{aligned}
\end{equation}
where for the last inequality, we have used the trace inequality, in addition to Young's inequality. The estimates \eqref{eq:stability_source_bounds_1}, \eqref{eq:stability_source_bounds_2} and \eqref{eq:stability_contact_bound} are substituted into \eqref{eq:stability_almost_final}, and we choose $\varepsilon_2+\varepsilon_3+\varepsilon_4+\varepsilon_5<\alpha_B$, such that the corresponding terms can be absorbed into the left-hand side of \eqref{eq:stability_almost_final}. Finally, we apply the discrete Grönwall inequality on the terms $\varepsilon_1\Delta t\sum_{n=1}^N\lVert p_n^h \rVert_{L^2(\Omega)}^2$ and $C_4\Delta t\sum_{n=1}^N\lVert \bu_n^h \rVert_{H^1(\Omega)}^2$; this requires us to absorb the $N$-th terms into the left-hand side of \eqref{eq:stability_almost_final}, and hence we require
\begin{align}
    \frac{c_0}{2}-\varepsilon_1\Delta t&>0,\\ \label{eq:idk}
    \frac{\alpha_A}{2}-C_4\Delta t&>0.
\end{align}
The first of these can always be satisfied by choosing a sufficiently small $\varepsilon_1$. However, the second inequality requires more careful consideration, since $C_4$ is not chosen freely, but rather, it depends on the value of $\varepsilon_4$ by Young's inequality; specifically, we have $C_4=(4\varepsilon_4)^{-1}$. Considering that we require $\varepsilon_4<\alpha_B$, we obtain from \eqref{eq:idk} the following requirement on the time step size:
\begin{equation} \label{eq:time_step_stability}
    \Delta t<2\alpha_A\alpha_B.
\end{equation}
Thus, if the time step fulfills \eqref{eq:time_step_stability}, we may use Grönwall's inequality to arrive at the estimate:
\begin{equation} \label{eq:stability_estimate_almost}
    \begin{aligned}
        &\lVert \bu_N^h \rVert_{H^1(\Omega)}^2+\lVert p_N^h \rVert_{L^2(\Omega)}^2+ \sum_{n=1}^N \Delta t \lVert \frunh \rVert_{H^1(\Omega)}^2+\sum_{n=1}^N \Delta t \lVert \bk^{1/2}\nabla p_n^h \rVert_{L^2(\Omega)}^2\\
        &\leq C \left(\lVert \bu_0^h \rVert_{H^1(\Omega)}^2+\lVert p_0^h \rVert_{L^2(\Omega)}^2+\Delta t\sum_{n=1}^N\left(\lVert \psi_n \rVert_{L^2(\Omega)}^2+\lVert \bm{f}_n\rVert_{L^2(\Omega)}^2+\lVert \bm{b}_n\rVert_{L^2(\Gamma_n)}^2+\lVert g \rVert_{L^2(\Gamma_c)}^2\right)\right),
    \end{aligned}
\end{equation}
where $C$ depends on the final time $T$, due to the Grönwall inequality, but is independent of $n, h$ and $\Delta t$. Since $\bu_0^h=\bu_0^I$ and $p_0^h=p_0^I$, the first two terms on the right-hand side can be bounded, up to a constant, by $\lVert\bu_0\rVert_{H^1(\Omega)}$ and $\lVert p_0\rVert_{L^2(\Omega)}$, by stability of the interpolation operators. \change{For the final four terms on the right-hand side, since the data is continuous on $[0,T]$, we can bound the Riemann sums, up to a constant, by a continuous-in-time norm. For instance,
\begin{equation}
    \Delta t \sum_{n=1}^N \lVert \psi_n \rVert^2_{L^2(\Omega)} \leq T\lVert \psi \rVert^2_{L^{\infty}(0,T;L^2(\Omega))}
\end{equation}}
Hence, the entire right-hand side can be bounded by a constant independent of $n, h$ and $\Delta t$.
\end{proof}

\begin{rmrk} \label{rem:time_step_size}
    The restriction on the time step size comes from the nonlinearity of the contact mechanics term \eqref{eq:stability_contact_bound}, where there are two terms that must be absorbed into the left-hand side of \eqref{eq:stability_almost_final}, but the size of both terms cannot be controlled by the constants of Young's inequality simultaneously. Alternatively, we could have multiplied $\lVert \bu_n^h \rVert_{H^1(\Omega)}^2$ by $\varepsilon_4$ and $\lVert \delta_t \bu_n^h \rVert_{H^1(\Omega)}^2$ by $C_4$; it can be shown that this would result in the exact same restriction \eqref{eq:time_step_stability} on the time step size.
\end{rmrk}

\subsection{Error estimate} \label{sec:error_estimates}

In this section, we derive an a priori error estimate between the fully discrete equations \eqref{eq:discrete_mechanics}, \eqref{eq:discrete_biot} and the continuous equations \eqref{eq:weak_mechanics}, \eqref{eq:weak_biot}. We first recall some standard estimates on the temporal error introduced by finite differences, using Taylor expansions:
\begin{lmm} \label{lemma:time_derivative_error}
    If $v \in H^2(0,T;X)$, where $X$ is either $L^2(\Omega)$ or $H^1(\Omega)$, then the error between the time derivative $\dot{v}_n$ and the backwards finite difference approximation $\delta_t v_n$, summed over all time steps $t_n$, satisfies:
    \begin{equation} \label{eq:taylor_bound_1}
        \sum_{n=1}^N\left\lVert\dot{v}_n-\delta_t v_n \right\rVert_X^2 \leq C\Delta t\lVert \ddot{v} \rVert_{L^2(0,T;X)}^2,
    \end{equation}
    where $C>0$ is a constant depending on $\Omega$. Moreover, we also have the following bound for $v \in H^1(0,T;X):$
    \begin{equation} \label{eq:taylor_bound_2}
        \sum_{n=1}^N\lVert  v_n-v_{n-1} \rVert_X^2 \leq \Delta t\lVert \dot{v} \rVert^2_{L^2(0,T;X)}. 
    \end{equation}
\end{lmm}
The proof is skipped, as these estimates are straightforward consequences of Taylor's theorem with integral remainder.
\begin{thrm} \label{thm:error_estimate}
    If \change{$\bu \in H^2(0,T;\bm{H}^{r+1}(\Omega))$} and $p \in H^1(0,T;H^{s+1}(\Omega)) \cap H^2(0,T;L^2(\Omega))$, and the time step size satisfies $\Delta t < 2 \alpha_A \alpha_B$, then the sequence of solutions $(\bu_n^h,p_n^h)$, $1 \leq n \leq N$, to problem (Q-D), satisfies:
        \begin{equation} \label{eq:asymptotic_estimate}
        \begin{aligned}
            \max_{1\leq n \leq N}\lVert p_n-p_n^h\rVert_{L^2(\Omega)}^2 &+ \max_{1 \leq n \leq N}\lVert \bu_n-\bu_n^h\rVert_{H^1(\Omega)}^2\\
            &+\sum_{n=1}^N \Delta t\lVert \bk^{1/2}\nabla (p_n-p_n^h) \rVert_{L^2(\Omega)}^2 \leq C((\Delta t)^2 + h^{r} \change{+ h^{2r}}+h^{2s}),
        \end{aligned}
    \end{equation}
    where $C$ depends on the final time $T$, but is independent of $n, h$ and $\Delta t$.
\end{thrm}
\begin{proof}
The proof is split up into three steps. We first derive an error equality between the continuous and discrete mass balance equations, then we derive an error inequality between the continuous and discrete variational inequalities, and finally the two estimates are combined.

\textbf{Step 1: Flow}

Our strategy for the mass balance equations closely follows the papers of Girault et al. \cite{poroelastic_notmixed,Girault2019}, in which error estimates were derived for the Biot equations generalized to a fractured porous medium. The discrete mass balance equation \eqref{eq:discrete_biot} is subtracted from the continuous equation \eqref{eq:weak_biot} evaluated at time $t_n$. Setting $q=q^h$, we obtain:
\begin{equation} \label{eq:biot_estimate_1}
    \left(c_0(\dot{p}_n-\frpnh)+\alpha\nabla \cdot (\dotu_n-\frunh),q^h\right)+\frac{1}{\mu}\left(\bk\nabla(p_n-p_n^h),\nabla q^h\right)=0.
\end{equation}
For convenience, we will add and subtract backwards finite differences of the exact solutions at $t_n$, thus splitting the error into a spatial discretization error and a temporal discretization error. We add and subtract $\delta_t p_n$ to the first term, and $\delta_t \bu_n$ to the second term, and we move the temporal discretization error to the right-hand side:
\begin{equation} \label{eq:biot_estimate_2}
\begin{aligned}
    \left(c_0\delta_t(p_n-p_n^h)+\alpha\nabla \cdot \delta_t(\bu_n-\bu_n^h),q^h\right)&+\frac{1}{\mu}\left(\bk\nabla(p_n-p_n^h),\nabla q^h\right)\\
    &=c_0\left(\delta_t p_n-\dotp_n,q^h\right)+\alpha\left(\nabla \cdot (\frun-\dotu_n),q^h\right).
\end{aligned}
\end{equation}
The spatial discretization errors, $\bu_n-\bu_n^h$ and $p_n-p_n^h$, are further split into interpolation and auxiliary discretization errors, by adding and subtracting $\bu_n^I$ and $p_n^I$. We define
\begin{equation} \label{eq:error_split}
\begin{aligned}
    \inu=\bu_n-\bu_n^I, \ \ \enu=\bu_n^I-\bu_n^h, \ \ \inp=p_n-p_n^I, \ \ \enp=p_n^I-p_n^h.
\end{aligned}
\end{equation}
Using the above notation, \eqref{eq:biot_estimate_2} can be rewritten as
\begin{equation} \label{eq:flow_estimate_added_2}
\begin{aligned}
    \left(c_0\delta_t(\inp+\enp)+\alpha\nabla \cdot \delta_t(\inu+\enu),q^h\right)&+\frac{1}{\mu}\left(\bk\nabla(\inp+\enp),\nabla q^h\right)\\
    &=c_0\left(\delta_t p_n-\dotp_n,q^h\right)+\alpha\left(\nabla \cdot (\frun-\dotu_n),q^h\right).
\end{aligned}
\end{equation}
We choose $q^h=\enp$ and use the discrete chain rule \eqref{eq:discrete_chain_rule} to rewrite the term $\left(\delta_t \enp,\enp\right)$. After moving all terms involving interpolation errors over on the right-hand side, we obtain:
\begin{equation} \label{eq:error_flow_final}
    \begin{aligned}
        \frac{c_0}{2\Delta t}\left(\lVert \enp \rVert_{L^2(\Omega)}^2-\lVert e_{n-1}^p\rVert_{L^2(\Omega)}^2+\lVert \enp-e_{n-1}^p\rVert_{L^2(\Omega)}^2\right)&+\alpha\left(\nabla \cdot \delta_t\enu,\enp\right)\\
        &+\frac{1}{\mu}\lVert \bk^{1/2}\nabla \enp \rVert_{L^2(\Omega)}^2=F_n^*.
    \end{aligned}
\end{equation}
\change{where we have defined:
\begin{equation}
    F_n^* = -c_0\left(\delta_t \inp,\enp\right)-\alpha\left(\nabla \cdot \delta_t\inu,\enp\right)-\frac{1}{\mu}\left(\bk\nabla\inp,\nabla\enp\right)+c_0\left(\delta_t p_n-\dotp_n,\enp\right)+\alpha\left(\nabla \cdot (\delta_t \bu_n-\dotu_n),\enp\right).
\end{equation}}
\textbf{Step 2: Mechanics}

Next, we turn to the variational inequalities. Here, our strategy is based on the analysis done in chapter 11 of the book of Sofonea and Han \cite{han2002quasistatic}. We add the continuous variational inequality \eqref{eq:weak_mechanics}, evaluated at $t_n$, to the discrete \eqref{eq:discrete_mechanics} variational inequality, and set $\bv=\delta_t \bu_n^h$:
\begin{equation} \label{eq:inequalities_added_1}
    \begin{aligned}
        &A[\bu_n,\frunh-\dotu_n]+A[\bu_n^h,\bv^h-\frunh]+B[\dotu_n,\frunh-\dotu_n]+B[\frunh,\bv^h-\frunh]\\
        &-\alpha\left(p_n,\nabla \cdot (\frunh-\dotu_n)\right)-\alpha\left(p_n^h,\nabla \cdot (\bv^h-\frunh)\right)\\
        &+J[\bu_n,\frunh]-J[\bu_n,\dotu_n]+J[\bu_n^h,\bv^h]-J[\bu_n^h,\frunh]\geq \left\langle\bm{f}^*_n,\frunh-\dotu_n+\bv^h-\frunh\right\rangle.
    \end{aligned}
\end{equation}
We aim to rewrite \eqref{eq:inequalities_added_1} to obtain terms containing the errors $\bu_n-\bu_n^h, \dotu_n-\frunh$ and $p_n-p_n^h$. Doing so will lead to an additional residual-type term, which is typical for variational inequalities \cite{han2002quasistatic,kikuchi_oden}. Following \cite{han2002quasistatic}, the terms on the left-hand side of \eqref{eq:inequalities_added_1} are rewritten to obtain:
\begin{equation} \label{eq:inequalities_added_2}
    \begin{aligned}
        &A[\bu_n-\bu_n^h,\dotu_n-\frunh]+B[\dotu_n-\frunh,\dotu_n-\frunh]-\alpha\left(p_n-p_n^h,\nabla \cdot (\dotu_n-\frunh)\right)\\
        &\leq A[\bu_n^h-\bu_n,\bv^h-\dotu_n]+B[\frunh-\dotu_n,\bv^h-\dotu_n]-\alpha\left(p_n^h-p_n,\nabla \cdot (\bv^h-\dotu_n)\right)\\
        &+J[\bu_n,\frunh]-J[\bu_n,\bv^h]+J[\bu_n^h,\bv^h]-J[\bu_n^h,\frunh]+R_n(\bv^h),
    \end{aligned}
\end{equation}
where the residual-type term $R_n$ is defined as follows:
\begin{equation} \label{eq:residual}
\begin{aligned}    R_n(\bv^h) = A&[\bu_n,\bv^h-\dotu_n]+B[\dotu_n,\bv^h-\dotu_n]-\alpha\left(p_n,\nabla \cdot (\bv^h-\dotu_n)\right)\\
&+J[\bu_n,\bv^h]-J[\bu_n,\dotu_n]-\left\langle\bm{f}^*_n,\bv^h-\dotu_n\right\rangle.
\end{aligned}
\end{equation}
As for the mass balance equation, we split the error into temporal and spatial discretization errors, and the latter is further split into an interpolation and auxiliary discretization error. More specifically, we write, using the same notation \eqref{eq:error_split} as before:
\begin{align}
    \bu_n-\bu_n^h &= \bu_n-\bu_n^I + \bu_n^I - \bu_n^h = \inu + \enu \\
    \dotu_n-\frunh &= \dotu_n - \frun + \frun - \frunh = \dotu_n - \frun + \franu + \frenu.
\end{align}
Moreover, we set $\bv^h=\dotu_n^I$, and use the notation $\dot{a}_n^{\bu}=\dotu_n-\dotu_n^I$. This is a natural choice of test function, as $\dot{a}_n^{\bu}$ will satisfy the interpolation error estimate \eqref{eq:interpolation_estimate} given sufficient regularity of $\dotu$. After an elementary but lengthy calculation, we can rewrite \eqref{eq:inequalities_added_2} to
\begin{equation} \label{eq:final_mechanics_equation}
    \begin{aligned}
        &\frac{1}{2\Delta t}\left(A[\enu,\enu]-A[e_{n-1}^{\bu},e_{n-1}^{\bu}]+A[\enu-e_{n-1}^{\bu},\enu-e_{n-1}^{\bu}]\right)+B[\frenu,\frenu]-\alpha\left(\enp,\change{\nabla \cdot }\frenu\right)\\
        &\leq A_n^* + B_n^* + C_n^* + J_n^* + R_n(\dotu_n^I),
    \end{aligned}
\end{equation}
where
\begin{equation}
    A^*_n = A[\inu+\enu,\dot{a}_n^{\bu}] -A[\enu,\dotu_n-\frun+\franu]-A[\inu,\dotu_n-\frun+\franu+\frenu],
\end{equation}
\begin{equation}
\begin{aligned}
    B^*_n = B&[\franu+\frenu+\dotu_n-\frun,\dot{a}_n^{\bu}]-B[\dotu_n-\frun,\dotu_n-\frun]-B[\franu,\franu]\\
        &-2B[\dotu_n-\frun,\franu]-2B[\frenu,\franu+\dotu_n-\frun],
\end{aligned}
\end{equation}
\begin{equation}
    C_n^* = \alpha\left(\inp,\nabla \cdot (\dotu_n-\frun+\franu+\frenu)\right)+\alpha\left(\enp,\nabla \cdot (\dotu_n-\frun+\franu)\right)-\alpha\left(\inp+\enp,\dot{a}_n^{\bu}\right),
\end{equation}
\begin{equation}
    J_n^* = J[\bu_n,\frunh]-J[\bu_n,\dotu_n^I]+J[\bu_n^h,\dotu_n^I]-J[\bu_n^h,\frunh],
\end{equation}
and the residual term $R_n$ is defined as before, but with $\bv^h=\dotu_n^I$ substituted. Note that, just like \eqref{eq:error_flow_final}, we have moved all terms involving interpolation operators to the right-hand side of \eqref{eq:final_mechanics_equation}.

\textbf{Step 3: Combining flow and mechanics}

The next step is then to combine \eqref{eq:error_flow_final} and \eqref{eq:final_mechanics_equation}. Noting that the coupling term $\alpha(\enp,\change{\nabla \cdot}\frenu)$ cancels, we obtain:
\begin{equation} \label{eq:big_estimate}
    \begin{aligned}
        &\frac{1}{2\Delta t}\left(A[\enu,\enu]-A[e_{n-1}^{\bu},e_{n-1}^{\bu}]+A[\enu-e_{n-1}^{\bu},\enu-e_{n-1}^{\bu}]\right)+B[\frenu,\frenu]\\
        &+\frac{c_0}{2\Delta t}\left(\lVert \enp \rVert_{L^2(\Omega)}^2-\lVert e_{n-1}^p\rVert_{L^2(\Omega)}^2+\lVert \enp-e_{n-1}^p\rVert_{L^2(\Omega)}^2\right)+\frac{1}{\mu}\lVert \bk^{1/2}\nabla \enp \rVert_{L^2(\Omega)}^2\\
        &\leq A_n^* + B_n^* + C_n^* + J_n^* + F_n^* + R_n(\dotu_n^I).
    \end{aligned}
\end{equation}
Henceforth, we will drop the terms $A[e_{n}^{\bu}-e_{n-1}^{\bu},e_{n}^{\bu}-e_{n-1}^{\bu}]$ and $\lVert \enp-e_{n-1}^p\rVert_{L^2(\Omega)}$ from the left-hand side of \eqref{eq:big_estimate}, as these terms are not needed for the upcoming analysis. Multiplying both sides of \eqref{eq:big_estimate} by $\Delta t$, summing both sides from $n=1$ to $N$, and using the coercivity \eqref{eq:v-ellipticity} of the bilinear forms, plus the fact that $e_0^p=0,e_0^{\bu}=\bm0$, due to our choice of initial conditions, we get
\begin{equation} \label{eq:big_estimate_summed}
    \begin{aligned}
        &\frac{\alpha_A}{2}\lVert e_N^{\bu}\rVert_{H^1(\Omega)}^2+\alpha_B\sum_{n=1}^N\Delta t\left\lVert \frenu \right\rVert_{H^1(\Omega)}^2+\frac{c_0}{2}\lVert e_N^p\rVert_{L^2(\Omega)}^2+\frac{1}{\mu}\sum_{n=1}^N\Delta t\lVert \bk^{1/2}\nabla \enp \rVert_{L^2(\Omega)}^2\\
    &\leq \Delta t \sum_{n=1}^N \lvert A_n^* \rvert + \lvert B_n^* \rvert + \lvert C_n^* \rvert + \lvert J_n^* \rvert + \lvert F_n^* \rvert + \lvert R_n(\dotu_n^I) \rvert.
    \end{aligned}
\end{equation}
We proceed to bound the terms on the right-hand side of \eqref{eq:big_estimate_summed}. In the upcoming estimates, we denote generic constants by $C$, while $\varepsilon_i$ denote constants of Young's inequality that are chosen freely (and they must be chosen sufficiently small, as will be shown). With the exception of one case, which will be treated separately, the constants involved are independent of the discretization parameters.

Starting with the $A_n^*$-term, we use the continuity of the bilinear form \eqref{eq:continuity_bilinear} and the triangle inequality to obtain
\begin{equation} \label{eq:error_comp_1}
    \begin{aligned}
        \Delta t\sum_{n=1}^N \lvert A_n^* \rvert
        \leq m_A\Delta t\sum_{n=1}^N &\left(\vphantom{\rule{0pt}{3ex}}\lVert \dot{a}_n^{\bu}\rVert_{H^1(\Omega)}\left(\lVert \inu \rVert_{H^1(\Omega)}+\lVert \enu \rVert_{H^1(\Omega)}\right)\right.\\
        &+\lVert \enu \rVert_{H^1(\Omega)}\left( \lVert \dotu_n-\delta_t\bu_n\rVert_{H^1(\Omega)}+\lVert \delta_t \inu \rVert_{H^1(\Omega)}\right)\\
        &\left.+\lVert \inu \rVert_{H^1(\Omega)}\left( \lVert \dotu_n-\delta_t\bu_n\rVert_{H^1(\Omega)}+\lVert \delta_t \inu \rVert_{H^1(\Omega)}+\lVert \delta_t \enu \rVert_{H^1(\Omega)}\right)\vphantom{\rule{0pt}{3ex}}\right).\\
    \end{aligned}
\end{equation}
Using Young's inequality on each term on the right-hand side of \eqref{eq:error_comp_1} yields, after merging the various constants:
\begin{equation} \label{eq:error_comp_2}
\begin{aligned}
    \Delta t \sum_{n=1}^N \lvert A_n^* \rvert &\leq \varepsilon_1\sum_{n=1}^N \Delta t \lVert \delta_t \enu \rVert_{H^1(\Omega)}^2\change{+\varepsilon_2\Delta t \sum_{n=1}^N \lVert \enu \rVert_{H^1(\Omega)}^2}\\
    &+C\Delta t \sum_{n=1}^N \left(\lVert \delta_t \inu \rVert_{H^1(\Omega)}^2+\lVert \dotu_n-\delta_t \bu_n \rVert_{H^1(\Omega)}^2+\lVert \inu \rVert_{H^1(\Omega)}^2+\lVert \dot{a}_n^{\bu}\rVert_{H^1(\Omega)}^2\right).
\end{aligned}
\end{equation}
The first two terms inside the second sum on the right-hand side of \eqref{eq:error_comp_2} are bounded by the Taylor expansion formulas \eqref{eq:taylor_bound_1} and \eqref{eq:taylor_bound_2} as follows:
\begin{equation} \label{eq:taylor_estimate_1}
  \Delta t \sum_{n=1}^N\lVert \dotu_n-\frun\rVert_{H^1(\Omega)}^2 \leq C(\Delta t)^2 \lVert \ddot{\bu}\rVert_{L^2(0,T;H^1(\Omega))}^2,  
\end{equation}
\begin{equation} \label{eq:taylor_estimate_2}
    \Delta t\sum_{n=1}^N \lVert \franu \rVert_{H^1(\Omega)}^2=\sum_{n=1}^N \frac{1}{\Delta t}\lVert \inu-a_{n-1}^{\bu} \rVert_{H^1(\Omega)}^2 \leq \change{\lVert \dotu-\dotu^I\rVert_{L^2(0,T;H^1(\Omega))}^2},
\end{equation}
Hence, the final estimate for the $A_n^*$-term becomes:
\begin{equation} \label{eq:estimate_A}
    \begin{aligned}
            \Delta t \sum_{n=1}^N \lvert A_n^* \rvert &\leq \varepsilon_1\sum_{n=1}^N \Delta t \lVert \delta_t \enu \rVert_{H^1(\Omega)}^2\change{+\varepsilon_2\Delta t \sum_{n=1}^N \lVert \enu \rVert_{H^1(\Omega)}^2}+C\left(\vphantom{\sum_{1}^N}\change{\lVert \dotu-\dotu^I \rVert_{L^2(0,T;H^1(\Omega))}^2}\right.\\
            &+(\Delta t)^2\lVert \ddot{\bu} \rVert_{L^2(0,T;H^1(\Omega))}^2
            \left.+\Delta t\sum_{n=1}^N \left(\lVert \inu \rVert_{H^1(\Omega)}^2+\lVert \dot{a}_n^{\bu}\rVert_{H^1(\Omega)}^2\right)\right).
    \end{aligned}
\end{equation}
The terms $B_n^*$, $C_n^*$ and $F_n^*$ are estimated in the same way as the $A_n^*$-term, and so we will merely state these estimates without showing details of the derivations. Formulas analogous to \eqref{eq:taylor_estimate_1} and \eqref{eq:taylor_estimate_2} are used to bound terms of the form $\Delta t \sum_{n=1}^N\lVert \delta_t \inp \rVert_{L^2(\Omega)}^2$ and $\Delta t \sum_{n=1}^N\lVert \dotp_n-\delta_t p_n \rVert_{L^2(\Omega)}^2$. For terms involving the divergence operator, we use the fact that $\lVert \nabla \cdot \bv\rVert_{L^2(\Omega)} \leq C\lVert \bv \rVert_{H^1(\Omega)}$, for some constant $C$ independent of discretization parameters. The estimates are as follows:
\begin{equation} \label{eq:estimate_B}
    \begin{aligned}
        \Delta t \sum_{n=1}^N \lvert B_n^* \rvert \leq \change{\varepsilon_3}\Delta t \sum_{n=1}^N \lVert \delta_t \enu\rVert_{H^1(\Omega)}^2+C\left(\vphantom{\sum_1^N}\change{\lVert \dotu-\dotu^I\rVert^2_{L^2(0,T;H^1(\Omega)}}\right.&+(\Delta t)^2\lVert \ddot{\bu} \rVert_{L^2(0,T;H^1(\Omega))}^2\\
        &\left.+\Delta t\sum_{n=1}^N\lVert \dot{a}_n^{\bu}\rVert_{H^1(\Omega)}^2\right),
    \end{aligned}
\end{equation}
\begin{equation} \label{eq:estimate_C}
    \begin{aligned}
        \Delta t \sum_{n=1}^N \lvert C_n^* \rvert &\leq \change{\varepsilon_4}\Delta t\sum_{n=1}^N \lVert\delta_t \enu \rVert_{H^1(\Omega)}^2\change{+\varepsilon_5\Delta t\sum_{n=1}^N \lVert \enp \rVert_{L^2(\Omega)}^2}+C\left(\vphantom{\sum_1^N}\change{\lVert \dotu-\dotu^I\rVert^2_{L^2(0,T;H^1(\Omega))}}\right.\\
        &+(\Delta t)^2\lVert \ddot{\bu} \rVert_{L^2(0,T;H^1(\Omega))}^2
        \left.+\Delta t\sum_{n=1}^N \left( \lVert \inp \rVert_{L^2(\Omega)}^2+\lVert \dot{a}_n^{\bu}\rVert_{H^1(\Omega)}^2\right)\right),
    \end{aligned}
\end{equation}
\begin{equation} \label{eq:estimate_F}
    \begin{aligned}
        &\Delta t \sum_{n=1}^N \lvert F_n^* \rvert \leq \change{\varepsilon_6}\Delta t \sum_{n=1}^N \lVert \bk^{1/2}\nabla \enp \rVert_{L^2(\Omega)}^2\change{+\varepsilon_7 \Delta t \sum_{n=1}^N\lVert \enp \rVert_{L^2(\Omega)}^2}+C\left(\vphantom{\sum_1^N }\change{\lVert \dotp-\dotp^I \rVert_{L^2(0,T;L^2(\Omega))}^2}\right.\\
        &+\change{\lVert \dotu-\dotu^I \rVert_{L^2(0,T;H^1(\Omega))}^2}
        \left.+(\Delta t)^2\left(\lVert \ddot{p}\rVert_{L^2(0,T;L^2(\Omega))}^2+\lVert \ddot{\bu} \rVert_{L^2(0,T;H^1(\Omega))}^2\right)+\Delta t \sum_{n=1}^N \lVert \bk^{1/2}\nabla \inp \rVert_{L^2(\Omega)}^2\right).
    \end{aligned}
\end{equation}
Next, we consider the $J_n^*$-term. It is bounded by Lemma \ref{lemma:friction_estimate}, and by writing $\dotu_n^I-\frunh=\dotu_n^I-\dotu_n+\dotu_n-\frun+\frun-\frunh$: 
\begin{equation} \label{eq:j_estimate}
\begin{aligned}
    \lvert J_n^*\rvert &= \left\lvert J[\bu_n,\frunh]-J[\bu_n,\dotu_n^I]+J[\bu_n^h,\dotu_n^I]-J[\bu_n^h,\frunh] \right\rvert\\
    &\leq C_J \lVert \bu_n-\bu_n^h \rVert_{H^1(\Omega)}\lVert \dotu_n^I-\frunh \rVert_{H^1(\Omega)}\\
    &\leq C_J\lVert \inu+\enu \rVert_{H^1(\Omega)}(\lVert \dotu_n^I-\dotu_n\rVert_{H^1(\Omega)}+\lVert \dotu_n-\frun\rVert_{H^1(\Omega)}+\lVert \frun-\frunh\rVert_{H^1(\Omega)})\\
    &= C_J\lVert \inu+\enu \rVert_{H^1(\Omega)}(\lVert \dot{a}_n^{\bu}\rVert_{H^1(\Omega)}+\lVert \dotu_n-\frun\rVert_{H^1(\Omega)}+\lVert \franu+\frenu\rVert_{H^1(\Omega)}).
\end{aligned}
\end{equation} 
From here we perform the same kinds of computations as for the previous terms to get
\begin{equation} \label{eq:estimate_J}
    \begin{aligned}
      & \Delta t \sum_{n=1}^N \lvert J_n^* \rvert \leq \change{(\varepsilon_8 + \varepsilon_9)} \sum_{n=1}^N\Delta t \lVert \delta_t \enu \rVert_{H^1(\Omega)}^2 \change{+((4\varepsilon_8)^{-1}+\varepsilon_{10})\sum_{n=1}^N \Delta t \lVert \enu \rVert_{H^1(\Omega)}^2}\\
      &+C\left(\vphantom{\sum_1^N}\change{\lVert \dotu-\dotu^I \rVert_{L^2(0,T;H^1(\Omega))}^2}+(\Delta t)^2\lVert \ddot{\bu} \rVert_{L^2(0,T;H^1(\Omega))}^2+\Delta t\sum_{n=1}^N\left(\lVert \inu \rVert_{H^1(\Omega)}^2+\lVert \dot{a}_n^{\bu}\rVert_{H^1(\Omega)}^2\right)\right).
    \end{aligned}
\end{equation}
\change{As in the stability estimate, due to the nonlinearity of $J_n^*$, we get the term $(4\varepsilon_8)^{-1}\sum_{n=1}^N \Delta t \lVert \enu \rVert_{H^1(\Omega)}^2$ whose size can only be controlled by imposing a restriction on the time step size. This leads to the same restriction as in the stability estimate, $\Delta t < 2 \alpha_A \alpha_B$, required to employ Grönwall's inequality at the end of the proof.}

Finally, we consider the residual term $R_n$. By simple usage of Cauchy-Schwarz, continuity of the bilinear forms, and an estimate on the $J$-terms similar to Lemma \ref{lemma:friction_estimate}, we can bound the term by
\begin{equation}
    \begin{aligned}
        \lvert R_n(\dotu_n^I)\rvert &\leq \left\lvert A[\bu_n,\dot{a}_n^{\bu}]\right\rvert+\left\lvert B[\dotu_n,\dot{a}_n^{\bu}]\right\rvert+\alpha\left\lvert\left(p_n,\nabla \cdot \dot{a}_n^{\bu}\right)\right\rvert+\left\lvert J[\bu_n,\dotu_n^I]-J[\bu_n,\dotu_n]\right\rvert+\left\lvert\left\langle\bm{f}^*_n,\dot{a}_n^{\bu}\right\rangle\right\rvert\\
        &\leq C_n \lVert \dot{a}_n^{\bu}\rVert_{H^1(\Omega)},
    \end{aligned}
\end{equation}
\change{where $C_n$ depends on the continuous solution and data at time $t_n$. Hence,} this constant depends on the temporal discretization, which is in contrast to the previous constants, which have all been independent of discretization parameters. However, our assumptions on the temporal regularity of the solution and the data are sufficient to ensure that $C_n$ is bounded on $[0,T]$, and we can define $C$ to be any constant such that $\max_{1\leq n \leq N}C_n<C$ for all $N$, and hence obtain the estimate:
\begin{equation}
    \Delta t \sum_{n=1}^N \lvert R_n(\dotu_n^I)\rvert \leq C \Delta t \sum_{n=1}^N \lVert \dot{a}_n^{\bu}\rVert_{H^1(\Omega)},
\end{equation}
Note that this term does not appear in squared form, which will result in a reduced spatial convergence order.

The $\varepsilon_i$ constants are now chosen such that the corresponding terms can be absorbed into the left-hand side of \eqref{eq:big_estimate_summed}. \change{We accordingly choose $\varepsilon_1+\varepsilon_3+\varepsilon_4+\varepsilon_8+\varepsilon_9<\alpha_B$ and $\varepsilon_6<1/\mu$ and absorb the corresponding terms into the left-hand side of \eqref{eq:big_estimate_summed}. Finally, assuming the aforementioned restriction on the time step size, we employ Grönwall's inequality on the terms $(\varepsilon_2+(4\varepsilon_8)^{-1}+\varepsilon_{10})\Delta t \sum_{n=1}^N \lVert \enu \rVert_{H^1(\Omega)}^2$ and $(\varepsilon_5+\varepsilon_7)\Delta t\sum_{n=1}^N \lVert \enp \rVert_{L^2(\Omega)}^2$.} The result is, after combining the various constants:
\begin{equation} \label{eq:final_estimate}
    \begin{aligned}
        &\lVert e_N^{\bu}\rVert_{H^1(\Omega)}^2+\sum_{n=1}^N\Delta t\left\lVert \frenu \right\rVert_{H^1(\Omega)}^2+\lVert e_N^p\rVert_{L^2(\Omega)}^2+\sum_{n=1}^N\Delta t\lVert \bk^{1/2}\nabla \enp \rVert_{L^2(\Omega)}^2\\
    & \leq C \left(\vphantom{\sum_{n=1}^N} \lVert \change{\dotu-\dotu^I\rVert_{L^2(0,T;H^1(\Omega))}^2} + \change{\lVert \dotp-\dotp^I\rVert_{L^2(0,T;L^2(\Omega))}^2}+ (\Delta t)^2\left(\lVert \ddot{\bu}\rVert_{L^2(0,T;H^1(\Omega))}^2+\lVert \ddot{p}\rVert_{L^2(0,T;L^2(\Omega))}^2\right)\right.\\
    &\ \ \ \ \ \ \ \ \ \ \ \ \ \ \left.  +\Delta t \sum_{n=1}^N \left(\lVert \dot{a}_n^{\bu}\rVert_{H^1(\Omega)}+\lVert \inu \rVert_{H^1(\Omega)}^2 + \lVert \dot{a}_n^{\bu} \rVert_{H^1(\Omega)}^2+\lVert \inp \rVert_{L^2(\Omega)}^2+\lVert \bk^{1/2}\nabla \inp \rVert_{L^2(\Omega)}^2\right)\right),
    \end{aligned}
\end{equation}
where $C$ depends on the final time $T$, but is independent of $h$ and $\Delta t$. \change{Like in the stability estimate, because of our assumptions on the temporal regularity of the solution, the Riemann sums on the right-hand side can be bounded, up to a constant, by a continuous-in-time norm. For instance:
\begin{equation}
    \Delta t \sum_{n=1}^N \lVert \dot{a}_n^{\bu}\rVert_{H^1(\Omega)} \leq T\lVert \dotu-\dotu^I \rVert_{L^{\infty}(0,T;H^1(\Omega))}.
\end{equation}
Finally, by employing the interpolation error estimates \eqref{eq:interpolation_estimate} on the relevant terms on the right-hand side of \eqref{eq:final_estimate}, we obtain the desired asymptotic error estimate \eqref{eq:asymptotic_estimate}.}
\end{proof}

\section{Numerical experiment} \label{sec:num_exp}

In this section, we present a numerical experiment, in which numerical convergence orders are computed are compared to the theoretical orders of Theorem \ref{thm:error_estimate}. We set our domain to be the unit square, $\Omega=(0,1) \times (0,1)$, with the boundary split up as follows: $\Gamma_d=\Gamma_p=[0,1]\times \{1\}$, $\Gamma_c=[0,1] \times \{0\}$, $\Gamma_n=\Gamma \setminus (\Gamma_d \cup \Gamma_c)$ and $\Gamma_f=\Gamma \setminus \Gamma_p$. The source terms are set to be $\bm{f}=(0,-1)^T$ and $\psi=1$, and we consider a simple flat obstacle, setting $g=0$. To induce frictional effects, we set an inhomogeneous Neumann condition on the left boundary, with the value $\bm{b}=(0.3,0)$, while the right boundary is free (homogeneous Neumann condition). Unitary values are chosen for the material parameters, with the exception of the friction coefficient, which will be varied between two values in order to simulate both sticking and slipping behavior. For simplicity, we choose linear finite element spaces for both the pressure and displacement. The numerical experiment has been implemented in version 0.10 of the open-source finite element software FEniCSx \cite{fenics4,fenics1,fenics3,fenics2}. \change{The code for the numerical experiment was developed with partial assistance of the generative AI tool ChatGPT. All suggestions by AI were carefully reviewed and verified by the corresponding author.}

To solve the discrete nonlinear systems, we have employed a penalty method, in which the Coulomb friction law is regularized, resulting in an algorithm similar to the one in Section 3.2 of \cite{SIMO199297}. In order to obtain a sufficiently accurate solution, a relatively high value of $10^{6}$ is used for the penalty parameter (the same value is also used for the constant $c_p$ of the normal compliance law \eqref{eq:regularized_normal_compliance}). The penalty regularization turns the discrete variational inequality \eqref{eq:discrete_mechanics} into an equality, and the resulting nonlinear equation system is solved by Newton's method, using absolute and relative tolerances of $10^{-6}$. In order to avoid the non-differentiable point of the absolute value function at zero, we replace $\lvert x \rvert$ by $\sqrt{x^2+\varepsilon}$ in the equations, for some small $\varepsilon$ (we have set $\varepsilon=10^{-10}$).  The positive cut function $(\cdot)_+$, which is also non-differentiable at zero, is rewritten as $(x)_+=(x+\lvert x\rvert)/2$, and the same regularization of the absolute value function is used here as well. The linear systems are solved by direct LU factorization.

We perform two separate convergence analyses, one with respect to the spatial discretization parameter $h$ and one with respect to the temporal discretization parameter $\Delta t$. As no analytical solution is available, we shall use a numerical solution produced with a finer discretization as our reference solution. For the spatial convergence analysis, this corresponds to a numerical solution produced on a sufficiently fine spatial grid, while for the temporal convergence analysis, our reference solution is produced by taking sufficiently small time steps. Denote the reference solution at time step $n$ by $\bu^{\text{ref}}_n, p_n^{\text{ref}}$. In accordance with our theoretical error estimate \eqref{eq:asymptotic_estimate}, we shall compute the errors of pressure and displacement as follows:
\begin{subequations} \label{eq:errors}
\begin{align}
    &\change{\text{Pressure error}} = \max_{1\leq n \leq N}\lVert p_n^{\text{ref}}-p_n^h\rVert_{L^2(\Omega)} + \left(\Delta t\sum_{n=1}^N \lVert \bk^{1/2}\nabla (p_n^{\text{ref}}-p_n^h) \rVert_{L^2(\Omega)}^2\right)^{1/2},\\ \label{eq:h1_displacement_error}
    &\change{\text{Displacement error}} = \max_{1 \leq n \leq N}\lVert \bu_n^{\text{ref}}-\bu_n^h\rVert_{H^1(\Omega)}.
\end{align}
\end{subequations}
For both the convergence analysis in space and time, we consider two cases,  one with a low friction coefficient ($\mu_{\text{fr}}=0.01$), and one with a high friction coefficient ($\mu_{\text{fr}}=10)$. In the first case, the low friction coefficient causes the body to slip, while in the second case, the friction bound is not reached, and the body sticks to the obstacle (see Figure \ref{fig:visualization_simulation} below).
\begin{figure}[ht]
    \centering 
    \includegraphics[width=\textwidth]{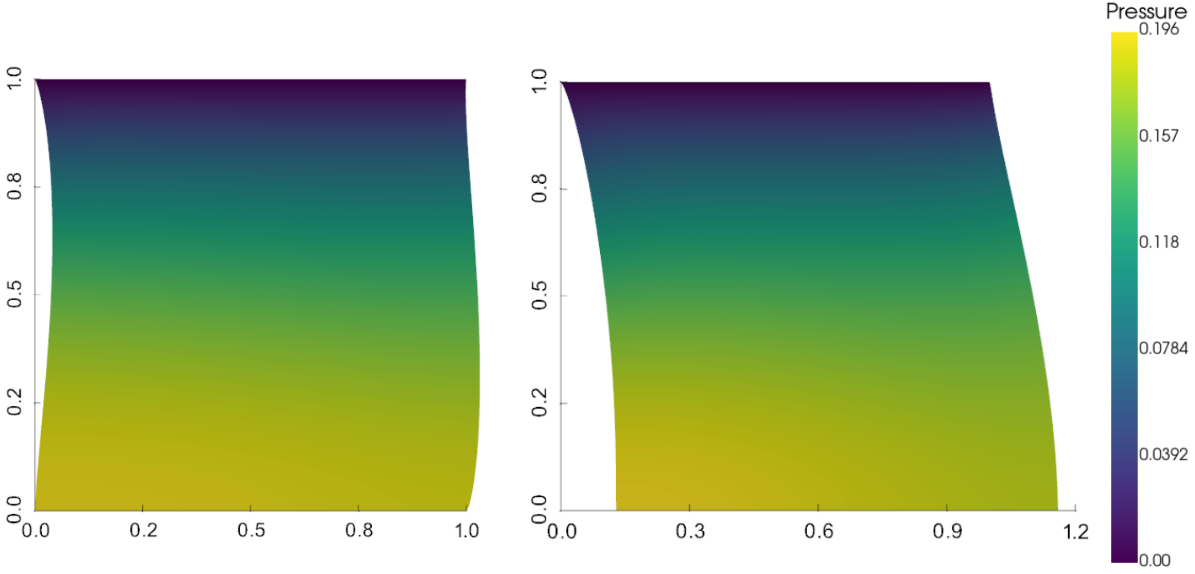}
    \caption{Visualization of the deformed body and pressure field using $\mu_{\text{fr}}=10$ (left) and $\mu_{\text{fr}}=0.01$ (right). The solution plotted here is the reference solution at $T=0.2$ for the spatial convergence analysis. The displacements are exaggerated (scaled by a factor of two) for visualization purposes.}
    \label{fig:visualization_simulation}
\end{figure}

We first perform the convergence analysis in space, by computing the errors of pressure and displacement for a sequence of nested, refined grids. More specifically, we consider grids with $1/h=5\cdot2^i, \ 0\leq i \leq 4$, and for our reference solution, we use $1/h=5\cdot2^6$. The errors \change{\eqref{eq:errors}} are computed by interpolating $p_n^{\text{ref}}, \bu_n^{\text{ref}}$ into the spaces in which $p_n^h,\bu_n^h$ belong. The simulations are run until $T=0.2$, employing four time steps of equal size $\Delta t=0.05$. The results are presented in Tables \ref{tab:errors_stick} and \ref{tab:errors_slip}, and in
Figure \ref{fig:h1_errors}. We observe first-order convergence of the pressure error, while the displacement error converges with a slightly lower order of around 0.8 in both cases. This is within the bounds of the theoretical error estimate \eqref{eq:asymptotic_estimate}, which for linear finite elements predicts a convergence order of 0.5 or higher for the displacement.
\newpage
\begin{table}[ht!]
\centering
\begin{tabular}{ | l | l | l | l | l | }
\hline
$1/h$ & \change{Pressure error} & Pressure rate & \change{Displacement error} & Displacement rate \\
\hline
5 & 0.00180281 & - & 0.00919235 & - \\
\hline
10 & 0.00237777 & -0.39936463 & 0.00461189 & 0.99507356 \\
\hline
20 & 0.00116485 & 1.02946132 & 0.00266417 & 0.7916744 \\
\hline
40 & 0.0005565 & 1.06569294 & 0.00157667 & 0.75680772 \\
\hline
80 & 0.00027269 & 1.02911958 & 0.00092292 & 0.77259447 \\
\hline
\end{tabular}
\caption{\change{Errors, computed by \eqref{eq:errors}}, and spatial convergence rates for the case with $\mu_{\text{fr}}=10$.}
\label{tab:errors_stick}
\end{table}

\begin{table}[ht!]
\centering
\begin{tabular}{ | l | l | l | l | l | }
\hline
$1/h$ & \change{Pressure error} & Pressure rate & \change{Displacement error} & Displacement rate \\
\hline
5 & 0.00132438 & - & 0.0180996 & - \\
\hline
10 & 0.0022577 & -0.76954073 & 0.00668837 & 1.43623213 \\
\hline
20 & 0.00112122 & 1.00978011 & 0.00362842 & 0.88231131 \\
\hline
40 & 0.00053864 & 1.05767593 & 0.00206753 & 0.81143111 \\
\hline
80 & 0.00026721 & 1.01135374 & 0.00119108 & 0.79563674 \\
\hline
\end{tabular}
\caption{\change{Errors, computed by \eqref{eq:errors}} and spatial convergence rates for the case with $\mu_{\text{fr}}=0.01$.}
\label{tab:errors_slip}
\end{table}

\begin{figure}[H]
    \centering
    \captionsetup[subfigure]{labelformat=empty}
    \begin{subfigure}{0.49\textwidth} 
    \includegraphics[width=\textwidth]{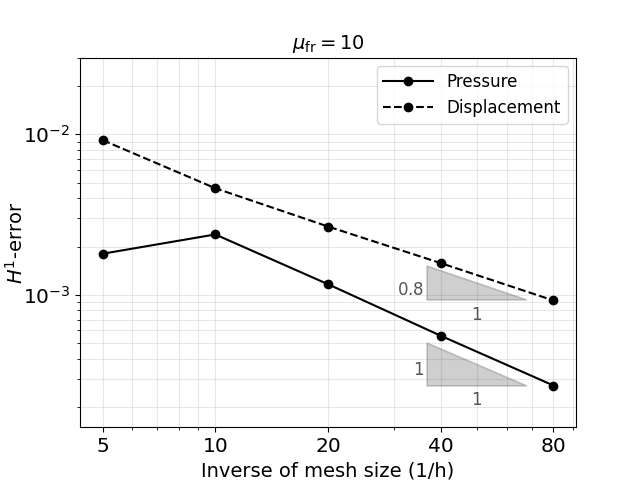}
    \caption{}
    \label{subfig:h1_errors_stick}
    \end{subfigure}
    \begin{subfigure}{0.49\textwidth} 
    \includegraphics[width=\textwidth]{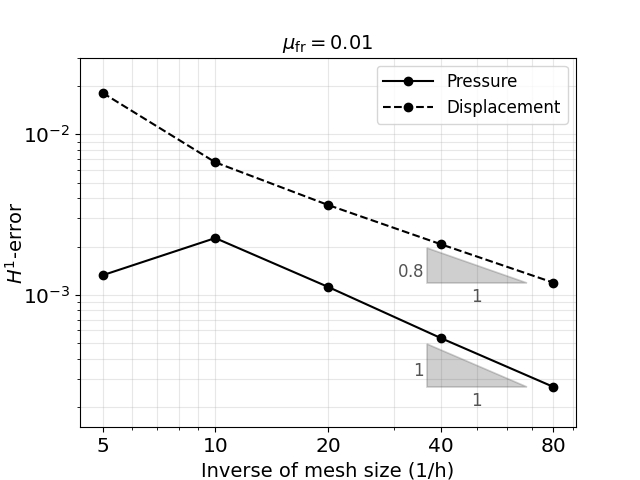}
    \caption{}
    \label{subfig:h1_errors_slip}
    \end{subfigure}
    \caption{\change{Errors} for various values of the spatial discretization parameter $h$, for the case with a high friction coefficient (left) and low friction coefficient (right).} 
    \label{fig:h1_errors}
\end{figure}
Next, we test the convergence order with respect to time, keeping the spatial mesh constant with $1/h=80$.
The simulations are run until the final time $T=2$, and the number of time steps $N$ are varied by setting $N=2^i, \ 0 \leq i \leq 5$. To produce our reference solution, we use $N=2^8$, and the errors \change{\eqref{eq:errors}} are computed by sampling the reference solution at the discrete time points it shares with the coarser solution. We remark that for smaller values of $T$, for instance 0.2 as used in the previous simulations, the errors would sometimes stagnate as the time step was lowered, indicating that the spatial error eventually dominates. Hence, we have chosen a larger value of $T$ for the analysis of the temporal error. The results are presented in Tables \ref{tab:errors_stick_time} and \ref{tab:errors_slip_time}, and in Figure \ref{fig:h1_errors_time}. These results indicate first-order convergence in time for both variables, as predicted by the theoretical error estimate \eqref{eq:asymptotic_estimate}.

\begin{table}[ht!]
\centering
\begin{tabular}{ | l | l | l | l | l | }
\hline
$T/\Delta t$ & \change{Pressure error} & Pressure rate & \change{Displacement error} & Displacement rate \\
\hline
1 & 0.17434072 & - & 0.03495026 & - \\
\hline
2 & 0.17883025 & -0.03668119 & 0.03022213 & 0.20969777 \\
\hline
4 & 0.1346983 & 0.40885912 & 0.01953443 & 0.62958605 \\
\hline
8 & 0.07744271 & 0.79853036 & 0.01117816 & 0.80533722 \\
\hline
16 & 0.0419831 & 0.88332081 & 0.00593149 & 0.91421611 \\
\hline
32 & 0.02076706 & 1.01551128 & 0.00293915 & 1.01299727 \\
\hline
\end{tabular}
\caption{\change{Errors, computed by \eqref{eq:errors}} and temporal convergence rates for the case with $\mu_{\text{fr}}=10$.}
\label{tab:errors_stick_time}
\end{table}

\begin{table}[ht!]
\centering
\begin{tabular}{ | l | l | l | l | l | }
\hline
$T/\Delta t$ & \change{Pressure error} & Pressure rate & \change{Displacement error} & Displacement rate \\
\hline
1 & 0.17355342 & - & 0.06048519 & - \\
\hline
2 & 0.17479689 & -0.01029969 & 0.06742579 & -0.15671852 \\
\hline
4 & 0.13188048 & 0.40644844 & 0.0429497 & 0.65065226 \\
\hline
8 & 0.07585641 & 0.79788801 & 0.02444939 & 0.81284997 \\
\hline
16 & 0.04125199 & 0.87880744 & 0.01287745 & 0.92495117 \\
\hline
32 & 0.02045156 & 1.0122525 & 0.0062945 & 1.0326823 \\
\hline
\end{tabular}
\caption{\change{Errors, computed by \eqref{eq:errors}} and temporal convergence rates for the case with $\mu_{\text{fr}}=0.01$.}
\label{tab:errors_slip_time}
\end{table}

\begin{figure}[H]
    \centering
    \captionsetup[subfigure]{labelformat=empty}
    \begin{subfigure}{0.49\textwidth} 
    \includegraphics[width=\textwidth]{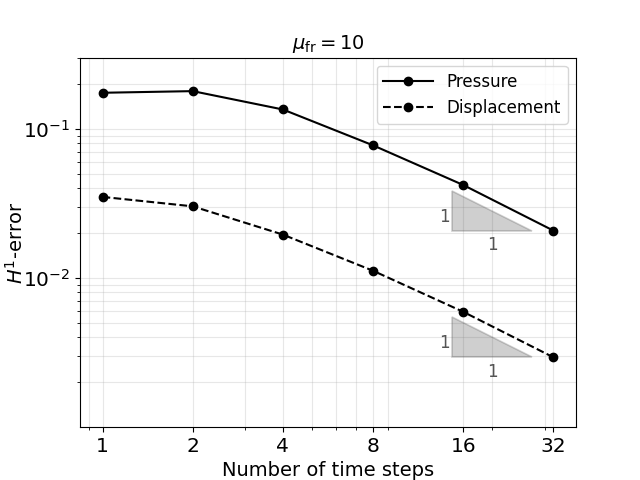}
    \caption{}
    \label{subfig:h1_errors_time_stick}
    \end{subfigure}
    \begin{subfigure}{0.49\textwidth} 
    \includegraphics[width=\textwidth]{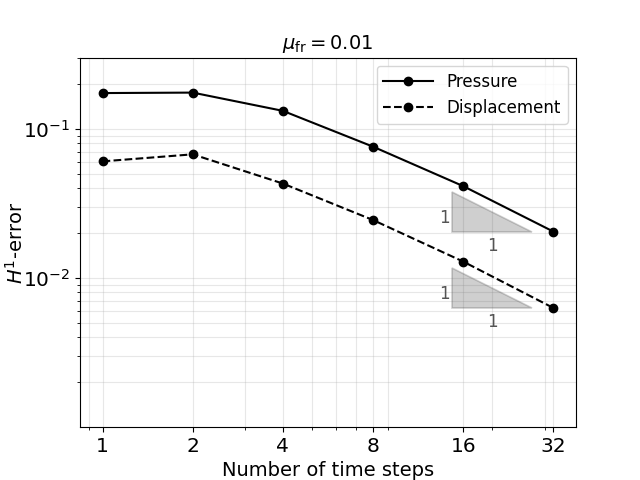}
    \caption{}
    \label{subfig:h1_errors_time_slip}
    \end{subfigure}
    \caption{\change{Errors} for various values of the time step size $\Delta t$, for the case with a high friction coefficient (left) and low friction coefficient (right).} 
    \label{fig:h1_errors_time}
\end{figure}

\section{Conclusion} \label{sec:conclusion}

We have performed a space-time numerical analysis of the Biot equations for a \change{poroviscoelastic} medium, subject to contact constraints including normal compliance and Coulomb friction. The resulting variational problem consists of a linear PDE coupled to a nonlinear variational inequality. This problem was discretized in space by conformal finite elements, and in time by the implicit Euler method. Our analysis consisted of proving existence and uniqueness of the discrete solution, deriving stability estimates of the discrete solution, and deriving a priori error estimates between the continuous and discrete solutions.

The existence proof was done by a novel method, in which the frictional contact problem was first reduced to the simpler Tresca friction model. Existence of a solution to this simpler problem was proven with the help of the well-known fixed-stress splitting scheme from poroelasticity. In the end, we obtain the existence to the problem with Coulomb friction by a fixed-point argument. The a priori error estimate was derived in \change{a combined space-time energy norm} for the sum of pressure and displacement. Optimal first-order convergence in time was obtained, while the order in space was reduced by a square root, which is typical for problems involving contact mechanics. Numerical experiments were conducted to verify the theoretical results, where convergence rates were computed separately in space and time, and separately for the pressure and displacement. First-order convergence in time were obtained for both variables. The order in space was optimal for the pressure variable, while the order for the displacement was slightly reduced. 

\section*{Funding}

This project has received funding from the European Research Council (ERC) under the European Union’s Horizon
2020 research and innovation program (grant agreement No. 101002507).

\section*{Data Availability Statement}

Runscripts for reproducing the results of the numerical experiments can be found in the
Docker image in \cite{zenodo_biot}.

\printbibliography

\end{document}